\documentclass[a4paper,leqno]{article}

\usepackage[latin1]{inputenc}
\usepackage{fancyhdr}
\usepackage[format=plain,labelfont=it,textfont=it]{caption}
\usepackage{needspace}
\usepackage{float}
\usepackage{tikz}
\usepackage{tkz-euclide}
\usetikzlibrary{decorations.pathreplacing}

\usepackage[pdftex]{lscape}
\usepackage{amsthm}
\usepackage{amsmath,amssymb}
\usepackage{enumerate,verbatim,color}
\usepackage{epsfig}
\usepackage{textcomp}
\usepackage{graphicx}
\usepackage{amsfonts}
\usepackage{mathrsfs}
\usepackage{upgreek}
\usepackage{mathtools}
\usepackage{pdfpages}
\usepackage{fullpage}
\usepackage{subfig}

\usepackage[right=2cm,left=2cm,top=2cm,bottom=2cm]{geometry}

\newcommand{\virgolette}[1]{``#1''}
\newtheorem{teorema}{Theorem}[section]

\newtheorem{lemma}[teorema]{Lemma}
\newtheorem{prop}[teorema]{Proposition}
\newtheorem{osss}[teorema]{Remark}

\theoremstyle{definition}

\theoremstyle{remark}

\newtheorem*{boundaries}{\bf Vertex boundary conditions}
\newtheorem*{notations}{\bf Operator boundary conditions}

\newcommand{\iii}{{\, \vert\kern-0.25ex\vert\kern-0.25ex\vert\, }}

\DeclareMathOperator{\spn}{span}
\newcommand{\ffi}{\varphi}

\newcommand{\NN}{\mathcal{N}}
\newcommand{\KK}{\mathcal{K}}

\newcommand{\Di}{\mathcal{D}}

\newcommand{\Z}{\mathbb{Z}}
\newcommand{\N}{\mathbb{N}}
\newcommand{\R}{\mathbb{R}}
\newcommand{\Q}{\mathbb{Q}}
\newcommand{\C}{\mathbb{C}}

\newcommand{\Hi}{\mathscr{H}}

\newcommand{\Gi}{\mathscr{G}}

\newcommand{\dd}{\partial}

\newcommand{\ra}{\rangle}
\newcommand{\la}{\langle}

\newcommand{\norm}[1]{\left\lVert #1 \right\rVert}

\date{}
\title{Exact controllability to eigensolutions of the bilinear heat equation\\ on compact networks}

\author{Piermarco Cannarsa
	\\ 
	{\small  University of Rome Tor Vergata}\\
	{\small  Via della Ricerca Scientifica 1, 00133 Roma} \\
	{\small \texttt{cannarsa@mat.uniroma2.it}}\\
	\\ 
	Alessandro Duca
	\\ 
	{\small  Universit\'{e} Paris-Saclay, UVSQ, Laboratoire de Math\'{e}matiques de Versailles. }\\
	{\small 45 avenue des \'{E}tats-Unis 78035 Versailles cedex, France} \\
	{\small \texttt{alessandro.duca@uvsq.fr}}\\
	\\ 
	Cristina Urbani
	\\ 
	{\small  University of Rome Tor Vergata}\\
	{\small  Via della Ricerca Scientifica 1, 00133 Roma} \\
	{\small \texttt{urbani@mat.uniroma2.it}}}

\begin{document}

\maketitle
\begin{abstract}
Partial differential equation on networks have been widely investigated in the last decades in view of their application to quantum mechanics (Schr\"odinger type equations) or to the analysis of flexible structures (wave type equations). Nevertheless, very few results are available for diffusive models despite an increasing demand arising from life sciences such as neurobiology. This paper analyzes the controllability properties of the heat equation on a compact network under the action of a single input bilinear control.

By adapting a recent method due to [Alabau-Boussouira, Cannarsa, Urbani, Exact controllability to eigensolutions for evolution equations of parabolic type via bilinear control], an exact controllability result to the eigensolutions of the uncontrolled problem is obtained in this work. A crucial step has been the construction of a suitable biorthogonal family under a non-uniform gap condition of the eigenvalues of the Laplacian on a graph.  Application to star graphs and tadpole graphs are included.

\vspace{3mm}


\vspace{2mm}
\end{abstract}

\section{Introduction}\label{intro}
In recent years, networks have attracted interest from a larger and larger part of the scientific community. This is certainly due to an increasing demand coming from numerous applications, but also to the new theoretical issues raised by the interaction between dynamics and network structure.

Applications of network type models, also called graphs, can be found, for instance, in quantum mechanics to study the dynamics of free
electrons in organic molecules (see the seminal work \cite{198} and also \cite{204,209}). In material science, they are often used to analyze the
superconductivity in granular and artificial materials \cite{121212}, or even in physics for acoustic and electromagnetic wave-guides networks (see \cite{112,181}). 

Another relevant application can be found in the study of structures composed by flexible/elastic strings or bars. When the elements of the structure are sufficiently thin, the system can be studied by considering wave type equations on graph domains. Many references on the subject are contained in the book \cite{wave}, where also various spectral and controllabilty results are provided (see also \cite{mugn,castro}).

Controllability of wave type equations on networks has been widely investigated in the last decades mostly for what concerns locally distributed or boundary controls. A non-exhaustive list on the subject is \cite{moran,tucavdo,dager,dagerzua}.  

We recall that a network, or graph, is a structure composed by edges connected in some points which are called vertices. The vertices are classified as external or internal: a vertex is external if it is the endpoint of only one edge, and it is internal otherwise. On such a structure it is possible to define a Hilbert space and a self-adjoint Laplace operator which satisfies the so--called Neumann-Kirchhoff boundary conditions at the internal vertices and Dirichlet or Neumann boundary conditions at the external vertices. For a wide analysis on graphs, we refer to the book \cite{berko}.

In this paper, we are interested in the bilinear control of networks -- or graphs -- of diffusive type. In other words, we consider parabolic type equations on each edge of the graph and we study the controllability of the system via suitable multiplicative controls. It is well known that, even on classical domains, important applications motivate the introduction of such kind of control actions that are more realistic in certain situations. For instance, in \cite{kap} several examples are given where the potential coefficient is used as a control.

A model described by our setting comes, for instance, from a neurobiological problem. From the seminal work by A. L. Hodgkin and A. F. Huxley \cite{hod-hux} it is known that the spread of electric current in neurons can be modeled by diffusive evolution equations. Therefore, Rall in \cite{Rall} proposed a cable model for a dendritic tree which consists of a cylindrical trunk and cylindrical branch components. This formulation leads to consider a cable equation of the form
\begin{equation}\label{cable}
    C_m\frac{\partial}{\partial t}v=\frac{a}{2R_i}\frac{\partial^2}{\partial x^2}v-gv
\end{equation}
on a network structure. Here $R_i,C_m,a$ and $g$ are positive constants representing, respectively, axoplasmic resistivity, membrane capacitance per unit area, fiber radius and membrane conductance. The variable $v$ represents the membrane potential at position $x$ along the cable at time $t$. It is assumed that at the nodes, where branching occurs, the potential is continuous and there is conservation of current flowing through the node  (Neumann-Kirchhoff conditions). At the terminals of the dendritic tree it is possible to impose open-end (Dirichlet) or closed-end (Neumann) conditions (see \cite{Abbott}). Equation \eqref{cable} describes the case of a passive cable with constant membrane conductance. However, to model realistic synaptic inputs we must allow the membrane capacity to depend on both space and time \cite[Paragraph 7]{Abbott}. Thus, the cable equation on the edges of the network becomes
\begin{equation*}
    \frac{\partial}{\partial t}v=\frac{\partial^2}{\partial x^2}v-g(x,t)v,
\end{equation*}
which is exactly of the form of our problem.

Bilinear control problems on networks were already studied for the Schr\"odinger equation in \cite{_ammari,_ammari1,_graphs,_graphs1}. However, fewer controllabilty results are available for diffusive networks. For instance, in \cite{bar-zua} the authors study the null controllability properties of a parabolic equation on networks, by means of additive controls. Nevertheless, there is a lack of controllability results for multiplicative controls and, in particular, for controls that are only a scalar function of time, even in standard domains.
This is certainly due to a structural obstacle, described in \cite{bms}, to control these systems \emph{along} a reference trajectory and also to the impossibility to apply methods based on the inverse mapping theorem developed for bilinear hyperbolic problems (see \cite{b,bl}).

On the other hand, the negative result in \cite{bms} does not prevent us from controlling parabolic systems \emph{to} a reference trajectory. Let us consider the abstract evolution equation 
\begin{equation}\label{control-pb-intro}
    \left\{\begin{array}{ll}
        y'(t)+Ay(t)+u(t)By(t)=0,&t\in(0,T)\\
        y(0)=y_0
    \end{array}\right.
\end{equation}
on a Hilbert space $X$. When $A:D(A)\subseteq X\to X$ is a densely defined, self-adjoint linear operator with compact resolvent, $B$ a bounded linear operator and $u$ is a bilinear control, the controllability of \eqref{control-pb-intro} was studied in \cite{acue,acu,cu}. Denoting by $\{\lambda_k\}_{k\in\N^*}$ the eigenvalues of the operator $A$, and by $\{\phi_k\}_{k\in\N^*}$ the corresponding normalized eigenfunctions, it is easy to see that the functions $\varphi_j(t)=e^{-\lambda_j t}\phi_j$ are solutions of the free dynamics ($u\equiv0$) with initial condition $y_0=\phi_j$. Such trajectories are called eigensolutions of problem \eqref{control-pb-intro}. In \cite{acue} it is proved that under the gap condition 
\begin{equation}\label{gap-intro}
    \sqrt{\lambda_{k+1}-\lambda_1}-\sqrt{\lambda_k-\lambda_1}\geq \gamma>0
\end{equation}
and the spreading assumption
\begin{equation}\label{ipB-intro}
    \langle B\phi_j,\phi_j\rangle\neq0,\quad|\lambda_k-\lambda_j|^q|\langle B\phi_j,\phi_k\rangle|\geq b,\quad\forall\,k\neq j
\end{equation}
with $p,q>0$, system \eqref{control-pb-intro} is locally controllable to the $j$-th eigensolution for any time $T>0$. Observe that, because of the gap condition \eqref{gap-intro}, the result is mostly applicable to low dimensional problems.

Nevertheless, since graphs are, essentially, one-dimensional domains, one can hope for an extension of the previous result to these structures.
Therefore, in this work we will focus our attention on the following control problem 
\begin{equation}\label{main}
    \left\{\begin{array}{ll}
        \partial_t\psi(t)-\Delta \psi(t)+u(t)B\psi(t)=0,&t\in(0,T)\\
        \psi(0)=\psi_0,
    \end{array}\right.
\end{equation}
defined on a graph, where $-\Delta$ is a Dirichlet-Neumann Laplacian equipped with suitable boundary conditions, $B$ is a bounded linear operator and the coefficient $u(\cdot)$ is a bilinear control.

However, when studying similar dynamics on a graph, one soon realizes that condition \eqref{gap-intro} fails. Indeed, what is always possible to prove for a Dirichlet-Neumann Laplacian on a general compact graph (see \cite{wave}) is a uniform gap condition among blocks of eigenvalues, namely,
\begin{equation*}
    \sqrt{\lambda_{k+M}}-\sqrt{\lambda_k}\geq\gamma>0,
\end{equation*}
for some $M\in\N^*$. Furthermore, for graphs with a specific structure (tadpole graphs, star graphs, double-ring graphs etc) it has been proved in \cite{_graphs} that, under suitable hypotheses on the length of the edges, also a weak gap condition between consecutive eigenvalues holds:
\begin{equation*}
    \sqrt{\lambda_{k+1}}-\sqrt{\lambda_k}\geq a_k,
\end{equation*}
with $a_k\asymp k^{-2}$.

Even so, we were able to prove the following controllability result for general compact graphs.
\begin{teorema}\label{thm-ex-control-graph}
Let $T>0$, $\Gi$ be a compact graph and $-\Delta$ be a Dirichlet-Neumann Laplacian on $L^2(\Gi,\R)$. Denoted by $\{\lambda_k\}_{k\in\N}$ the eigenvalues of $-\Delta$, and by $\{\phi_k\}_{k\in\N^*}$ the associated normalized eigenfunctios, assume that the following gap condition is satisfied
\begin{equation}\label{gap-required}
    \sqrt{\lambda_{k+1}}-\sqrt{\lambda_k}\geq a_k
\end{equation}
where $\{a_k\}_{k\in\N^*}$ is a sequence of positive real numbers such that $a_k \asymp k^{-p}$, with $p>0$, for $k\in\N^*$.

Let $B:X\to X$ be a bounded linear operator on $\Gi$ such that, for some $j\in\\N^*$
\begin{equation}\label{ipB}
    \la B\phi_j,\phi_1\ra\neq0, \text{ and }\,\, \exists\,b,q>0\,:\: \lambda_k^q\left|\la B\phi_j,\phi_k\ra\right|\geq b,\quad\forall\,k\neq 1.
\end{equation}
Then, system \eqref{main} is locally controllable to the $j$-th eigensolution $\varphi_j(t)=e^{-\lambda_k t}\phi_j$ in any positive time $T>0$, that is, there exists $u\in L^2(0,T)$ such that $\psi(T)=\varphi_j(T)$. Furthermore,
\begin{equation*}
\norm{u}_{L^2(0,T)}\leq \frac{e^{-\pi^2C/T}}{e^{2\pi^2C/(3T)}-1},
\end{equation*}
for a suitable positive constant $C$.
\end{teorema}
In Theorem \ref{thm-ex-control-graph}, the sequence $\{a_k\}_{k\in\N^*}$ is assumed to be polynomially decreasing. However, it is not difficult to prove that the same result is also valid for a wider class of spectral gaps as, for instance, $a_k\asymp e^{-\sqrt\lambda_k}$. Nevertheless, the weak gap condition of the explicit networks analyzed in this work is satisfied with $a_k\asymp k^{-p}$.

Another model giving rise to a fading gap condition is the one studied in \cite{boy-mor}. In this work the authors study the minimal null control time for problem with additive controls. They could overcome the phenomenon of condensation of eigenvalues in the moment problem they solve by a block resolution technique.

Even if the proof of our result is based on a linearization argument and the resolution of a moment problem as well, we are unable to use the approach of \cite{boy-mor} because no estimate on the cost of control with respect to control time $T$ is there provided. 

This paper is organized as follows.
\begin{itemize}
    \item In Section ,2 we introduce graph theory and recall some useful spectral properties of the Laplace operator on such structures,
    \item In Section 3, we introduce the notion of controllability to eigensolutions. Then, we prove Theorem \ref{thm-ex-control-graph} and we state two semi-global controllability results.
    \item Section 4 is devoted to explicit applications to specific graphs. Moreover we give a flavour of how to treat the case of multiple eigenvalues.
\end{itemize}

\section{Preliminaries}\label{Sec2}
\subsection{Introduction to graph theory}

We denote by graph a structure composed by segments or half-lines (named edges) connected each other in some 
points (named vertices). When the graph is composed only by a finite number of edges with finite length, it is
called compact graph. 

Let $\Gi$ be a compact graph composed by $N\in\N^*$ edges $\{e_j\}_{j\leq N}$ of lengths $\{L_j\}_{j\leq N}$ and 
$M\in\N^*$ vertices $\{v_j\}_{j\leq M}$. For every vertex $v$ of the graph $\Gi$, we denote 
\begin{align*}N(v):=\big\{l \in\{1,...,N\}\ |\ v\in 
e_l\big\}.\end{align*}
We call $V_e$ and $V_i$ the external and the internal vertices of 
the graph $\Gi$ (see Figure \ref{vertici}).
\begin{figure}[H]
	\centering
	\includegraphics[width=\textwidth-100pt, height=65pt]{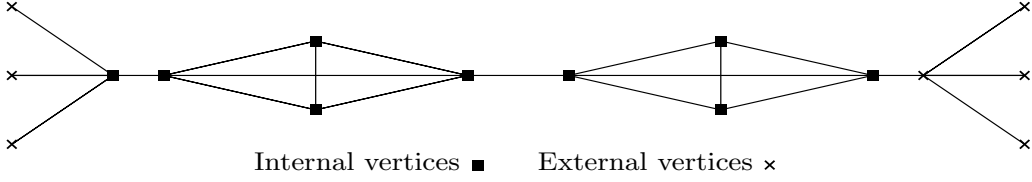}
	\caption{Internal and external vertices in a compact graph.}\label{vertici}
\end{figure}
We equip the graph with a metric parametrizing each edge $e_j$ with a coordinate going from $0$ to its 
length $L_j$. We 
consider functions $f:=(f^1,...,f^N)$ with domain the graph $\Gi$ such that $$f^j:(0,L_j)\rightarrow \R,\ \ \ j\leq N.$$ We introduce now the Hilbert space $L^2(\Gi):=L^2(\Gi,\R)=\prod_{j\leq N}L^2((0,L_j),\R)$ equipped with the norm $\|\cdot\|_{L^2}$ induced by the scalar product
$$\la\psi,\ffi\ra_{L^2}:=\sum_{j\leq N}\la\psi^j,\ffi^j\ra_{L^2(e_j,\R)}=\sum_{j\leq 
N}\int_{e_j}\psi^j(x)\ffi^j(x)dx,\ \ \ \  \ \ \forall\, \psi,\ffi\in L^2(\Gi,\R).$$
%For $s>0$, we define the spaces $$H^s=H^s(\Gi):=\prod_{j=1}^N 
%H^s(e_j,\R).$$
Let $f\in H^1$ and $v$ be a vertex of $\Gi$ connected once to an edge $e_j$ with 
$j\leq N$. When the coordinate parametrizing $e_j$ in the vertex $v$ is equal to $0$ (resp. $L_j$), we denote
\begin{align}\label{NK1}\dd_x f^j(v)=\dd_xf^j(0),\ \ \  \ \ \ \ \ \big(\text{resp.}\ \dd_x 
f^j(v)=-\dd_xf^j(L_j)\big).\end{align}
When $e_j$ is a loop and it is connected to $v$ in both of its extremes, we use the notation 
\begin{align}\label{NK2}\dd_x 
f^j(v)=\dd_xf^j(0)-\dd_xf^j(L_j).\end{align} When $v$ is an external vertex and $e_j$ is the only edge 
connected to $v$, we call $\dd_x f(v)=\dd_x f^j(v).$

Now, we are finally ready to characterize the self-adjoint Laplacian $-\Delta$ appearing in equation \eqref{main}, 
but before, we need to introduce the following boundary conditions on the vertices of the graph.

\needspace{3\baselineskip}
\begin{boundaries}{\ }
\begin{itemize}
\item[($\Di$)] A vertex $v\in V_e$ is equipped with Dirichlet boundary conditions when $f(v)=0$ for every $f\in D(-\Delta)$.

\item[($\NN$)] A vertex $v\in V_e$ is equipped with Neumann boundary conditions when $\dd_xf(v)=0$ for every 
$f\in D(-\Delta)$.
 \item[($\NN\KK$)] A vertex $v\in V_i$ is equipped with Neumann-Kirchhoff boundary conditions when every $f\in 
D(-\Delta)$ is 
continuous in $v$ and $\sum_{j\in N(v)}\dd_x f^j(v)=0$ (according to the notations \eqref{NK1} and \eqref{NK2}).

\end{itemize}
\end{boundaries}

Let us introduce the different definitions of $D(-\Delta)$ considered in this work.

\begin{notations}{\ }
\begin{itemize}

\item The operator $-\Delta$ on $\Gi$ is said to be a Dirichlet Laplacian when
$$D(-\Delta)=\{\psi\in H^2\ :\ \psi\text{ satisfies $(\NN\KK)$ in every $v\in V_i$ and ($\Di$) in 
each $v\in V_e$}\}.$$

\item The operator $-\Delta$ on $\Gi$ is said to be a Neumann Laplacian when
$$D(-\Delta)=\{\psi\in H^2\ :\ \psi\text{ satisfies $(\NN\KK)$ in every $v\in V_i$ and ($\NN$) in 
each $v\in V_e$}\}.$$

\item The operator $-\Delta$ on $\Gi$ is said to be a Dirichlet-Neumann Laplacian when
$$D(-\Delta)=\{\psi\in H^2\ :\ \psi\text{ satisfies $(\NN\KK)$ in every $v\in V_i$ and each $v\in V_e$ verifies ($\Di$) 
or 
($\NN$)}\}.$$
\end{itemize}
\end{notations}

\begin{osss}\label{compact_resolvent}
The boundary conditions described above imply that the Laplacian $-\Delta$ is self-adjoint (see \cite[Theorem\ 
3]{kuk}) and admits compact resolvent (see \cite[Theorem\ 18]{kuk}).\end{osss} 

\subsection{Some spectral properties}

We 
denote by 
$\{\lambda_k\}_{k\in\N^*}$ the ordered sequence of eigenvalues of $-\Delta$ and 
$\{\phi_k\}_{k\in\N^*}$ is a Hilbert basis of $L^2(\Gi,\R)$ composed by the corresponding eigenfunctions.

The following Lemma rephrases \cite[Lemma 6]{_graphs} (arXiv version \cite[Lemma 2.3]{_graphs}) which is a consequence of $\cite[Theorem  \ 
3.1.8]{berko}$ and 
$\cite[Theorem  \ 3.1.10]{berko}$. 

\begin{lemma}{  \cite[Lemma 6]{_graphs} (arXiv version \cite[Lemma 2.3]{_graphs}) }\label{interessante}
Let $\{\lambda_k\}_{k\in\N^*}$ be the eigenvalues of a Dirichlet-Neumann Laplacian $-\Delta$ defined on a compact graph. 
There holds $\lambda_k\asymp k^2$ for $k\geq 2$, {\it 
i.e.} there exist 
$C_1,C_2>0$ such that \begin{align*}C_1 k^2\leq \lambda_k\leq C_2 k^2,\ \ \ \ \ \ \ 
\forall\, k\geq 2.\end{align*}
\end{lemma}

When $-\Delta$ is not a Neumann Laplacian, we have
$0\not\in\sigma(-\Delta)$ (the spectrum of 
$-\Delta$), while when $A$ is a Neumann Laplacian we have $0\in\sigma(-\Delta)$ and $\lambda_1=0$.

From \cite[Proposition\ 6.2]{wave}, we have the following lemma.  
\begin{lemma}{\cite[Proposition\ 6.2]{wave}}\label{g13}
Let $\{\lambda_k\}_{k\in\N^*}$ be the ordered sequence of eigenvalues of a Dirichlet-Neumann Laplacian $-\Delta$ defined on a compact graph. 
There exist $\delta>0$ and
$M\in\N^*$ such that 
\begin{equation*}\begin{split}
\inf_{{k\in\N^*}}\left|\sqrt{\lambda_{k+M}}-\sqrt{\lambda_k}\right|\geq \delta M.\\
\end{split}\end{equation*}
\end{lemma}

Now, we present a spectral result valid for specific structures of graphs. The statement for the tadpole, the 
two-tails tadpole, the double-rings graph and the star graph with at most $4$ edges follows from \cite[Proposition 8]{_graphs} (arXiv version \cite[Proposition 2.5]{_graphs}). The case of the star graph with $N\in\N^*$ edges is proved in Appendix \ref{proof_gap_star} by using the techniques developed in \cite[Section 3]{_graphs1}.

\begin{prop}\label{gap_tadpole} Let $\Gi$ be either a tadpole, a 
two-tails tadpole, a double-rings graph or a star graph. Let $\{\lambda_k\}_{k\in\N^*}$ be the ordered sequence of eigenvalues of a Dirichlet-Neumann Laplacian $-\Delta$ defined on $\Gi$. If the set of the lengths of the edges $\{L_j\}_{j\leq N}\in (\R^+)^N$ is so that
$\big\{1,L_1,...,L_N\big\}$ are linearly independent over $\Q$ and all the ratios $L_k/L_j$ are algebraic irrational numbers, then, for every $\epsilon>0$, there exists $C>0$ such that
\begin{equation*}\sqrt{\lambda_{k+1}}-\sqrt{\lambda_k}\geq C  k^{-1-\epsilon},\ \ \ \ \ \ \ \ \forall\, k\in\N^*\end{equation*}
\end{prop}
\begin{proof}
See Appendix \ref{proof_gap_star}.
\end{proof}

\section{Abstract controllability results}

\label{Sec3}\subsection{Introduction to controllability to eigensolutions}

System \eqref{main} is a specific case of a more general control problem defined on a separable Hilbert space. Now, we recall the abstract result from \cite{acue} of local exact controllability of a bilinear parabolic problem of the form
\begin{equation}\label{u}
\left\{
\begin{array}{ll}
y'(t)+A y(t)+u(t)By(t)=0,& t>0,\\
y(0)=y_0.
\end{array}\right.
\end{equation}
Let $(X,\langle\cdot,\cdot\rangle,\norm{\cdot})$ be a separable Hilbert space and let $A:D(A)\subseteq X\to X$ be a densely defined linear operator such that
\begin{equation}
\label{ipA-intro}
\begin{array}{ll}
(a) & A \mbox{ is self-adjoint},\\
(b) &\exists\,\sigma\geq0\,:\,\langle Ax,x\rangle \geq -\sigma\norm{x} ^2,\,\, \forall\, x\in D(A),\\
(c) &\exists\,\lambda>-\sigma \mbox{ such that }(\lambda I+A)^{-1}:X\to X \mbox{ is compact}.
\end{array}
\end{equation}
When the assumptions above are verified, the spectrum of $A$ consists of a sequence of real numbers $\{\lambda_k\}_{k\in\N^*}$ which can be ordered, without loss of generality, as $-\sigma\leq\lambda_k\leq\lambda_{k+1}\to\infty$ as $k\to\infty$. We denote by $\{\phi_k\}_{k\in\N^*}$ the associated eigenfunctions, that is, $A\phi_k=\lambda_k\phi_k$ with $\norm{\phi_k}=1$, $\forall\,k\in\N^*$. Furthermore, $-A$ generates a strongly continuous semigroup denoted by $e^{-tA}$. The well-posedness of \eqref{u} is a classical result (see, i.e. \cite{bms}).
\begin{prop}
Let $T>0$. Let $A:D(A)\subseteq X\to X$ satisfy \eqref{ipA-intro}, $B:X\to X$ be a bounded linear operator, $u\in L^2(0,T)$. Then, for any $y_0\in X$ there exists a unique mild solution $y$ of \eqref{u}, that is, a function $y\in C^0([0,T];X)$ which satisfies
\begin{equation*}
    y(t;y_0,u)=e^{-tA}y_0-\int_0^t e^{-(t-s)A}u(s)By(s;y_0,u)ds,\quad\forall\,t\in[0,T].
\end{equation*}
Moreover, there exists a constant $C(T)>0$ such that
\begin{equation*}
    \sup_{t\in[0,T]}\norm{y(t;y_0,u)}\leq C(T)\norm{y_0}.
\end{equation*}
\end{prop}

Fixed any $j\in\N^*$, we introduce the $j$-th \emph{eigensolution} of problem \eqref{u}, that is the function $\varphi_j(t):=e^{-\lambda_j t}\phi_j$, which solves \eqref{u} for $u\equiv0$ and $y_0=\phi_j$. 

We recall that problem \eqref{u} is said to be locally controllable to the $j$-th eigensolution in time $T>0$ if there exists $R_{T}>0$ such that, for any $y_0\in X$ with $\norm{y_0-\varphi_j}<R_T$, there exists a control $u\in L^2(0,T)$ for which the solution of \eqref{u} verifies $y(T;y_0,u)=\varphi_j(T)$.

Given $T>0$ and $j\in\N^*$, we introduce the linear problem
\begin{equation}\label{lin-intro}
\begin{cases}
z'(t)+Az(t)+u(t)B\phi_j=0,\ \ \ \ \ \ \ \ &t\in(0,T),\\
z(0)=z_0.
\end{cases}
\end{equation}
We denote by $z(\cdot;z_0,u)$ the solution of \eqref{lin-intro} associated to the initial condition $z_0$ and control $u$. We say that the pair $\{A,B\}$ is \emph{$j$-null controllable} if there exists a constant $K(T)>0$ such that for any $z_0\in X$ there exist a control $u\in L^2(0,T)$ for which the solution of \eqref{lin-intro} satisfies $z(T;z_0,u)=0$ and moreover $\norm{u}_{L^2(0,T)}\leq K(T)\norm{z_0}$. In particular,
we call
\begin{equation*}
K(T):=\sup_{\norm{z_0}=1}\inf \left\{\norm{u}_{L^2(0,T)}\,:\,z(T;z_0,u)=0\right\}
\end{equation*}
the \emph{control cost}. When assumption \eqref{ipA-intro}, \eqref{gap-intro} and \eqref{ipB-intro} on the pair $\{A,B\}$ are verified, the authors of \cite{acue} proved the following controllability result.

\begin{teorema}\label{teo1}
Let $A:D(A)\subset X\to X$ be a densely defined linear operator such that \eqref{ipA-intro} holds and let $B:X\to X$ be a bounded linear operator such that \eqref{ipB-intro} holds. Let $\{A,B\}$ be $j$-null controllable in any $T>0$ for some $j\in\N^*$ and such that
\begin{equation}\label{bound-control-cost}
K(\tau)\leq e^{\nu/\tau},\quad\forall\,0<\tau\leq T_0,
\end{equation}
for some constants $\nu,T_0>0$.

Then, the system \eqref{u} is locally controllable to the $j$-th eigensolution in any positive time $T>0$. Moreover, the following estimate holds
\begin{equation}\label{u-bound}
\norm{u}_{L^2(0,T)}\leq \frac{e^{-\pi^2C/T}}{e^{2\pi^2C/(3T)}-1},
\end{equation}
for a suitable positive constant $C$.
\end{teorema}

\subsection{Proof of Theorem \ref{thm-ex-control-graph}}

This section is devoted to the proof of Theorem \ref{thm-ex-control-graph}.

Let $T>0$, $\Gi$ be a compact graph and let $X=L^2(\Gi,\R)$. We shall prove the controllability of the following problem
\begin{equation}\label{bil-par-eq}
    \begin{cases}
    \dd_t\psi(t)-\Delta\psi(t)+u(t)B\psi(t)=0,\ \ \ \ \ \ \ \ &t\in(0,T),\\
    \psi(0)=\psi_0.
    \end{cases}
\end{equation}
to the $j$-th eigensolution $\varphi_j(t)$ in time $T$.

The proof of Theorem \ref{thm-ex-control-graph} is based on the following propositions.

\begin{prop}\label{prop-bio-fam}
Let $\{\lambda_k\}_{k\in\N^*}$ be a sequence of ordered non-negative real numbers such that $\sum_{k=2}^\infty \lambda_k^{-1}<+\infty $.
Assume that there exist $M\in\N^*$, $\gamma>0$ and an ordered sequence of decreasing positive real numbers $\{a_k\}_{k\in\N^*}$ verifying, for every $k\in\N^*$,
\begin{equation}\label{gap-prop32}
\begin{split}
\begin{cases}
\sqrt{\lambda_{k+M}}-\sqrt\lambda_{k }\geq \gamma,\\\\
\sqrt{\lambda_{k+1}}-\sqrt\lambda_{k }\geq a_k.
\end{cases}
\end{split}
\end{equation}
Then, there exists a sequence of functions $\{\sigma_k\}_{k\in\N^*}$ which is biorthogonal to the family of exponentials $\{e^{\lambda_k t}\}_{k\in\N^*}$ in $L^2(0,T)$:
\begin{equation*}
\int_0^T \sigma_k(t)e^{\lambda_j t}dt=\delta_{k,j},\quad\forall\,k,j\in\N^*.
\end{equation*}
Moreover, the biorthogonal family $\{\sigma_k\}_{k\in\N^*}$ satisfies
\begin{equation}\label{biort-bound}
\norm{\sigma_k}^2_{L^2(0,T)}\leq C\Big (1+ \frac{\gamma^2}{a_{k}(a_{k} + 2   \sqrt\lambda_1)}\Big)^{2M}
e^{-2\lambda_k T}e^{C/(T\gamma^2)}e^{C\sqrt{\lambda_k}/\gamma}B(T,\gamma),
\end{equation}
where \begin{equation*}\begin{split}B(T,\gamma):=\begin{cases}\frac{1}{T}+\frac{1}{T^2\gamma^2},\ \ \ \ \ \ &T\leq\frac{1}{\gamma^2},\\\\
{C}\gamma^2,\ \ \ \ \ \ &T>\frac{1}{\gamma^2},\\
\end{cases}\end{split}\end{equation*}
\end{prop}

\begin{proof}
See Appendix \ref{bioApp}.
\end{proof}
Notice that under the hypotheses of Theorem \ref{thm-ex-control-graph}, the conditions in \eqref{gap-prop32} are satisfied. Indeed, the first inequality in \eqref{gap-prop32} is always valid for a Dirichlet-Neumann Laplacian $-\Delta$ on a graph thanks to Lemma \ref{g13}.

\begin{prop}\label{prop-G_R}
Let $-\Delta$ be a Dirichlet-Neumann Laplacian on $X$. Let $B:X\to X$ be a linear bounded operator on $\Gi$ such that \eqref{ipB} holds.

Let $p>0$ and $j\in\N^*$, we define the following function
\begin{equation}\label{G_R}
    G_{R,j}(T):=\frac{R}{T^2}e^{R/T}\sum_{k=1}^\infty\lambda_k^p\frac{e^{-2\lambda_kT+R\sqrt{\lambda_k}}}{|\langle B\phi_j,\phi_k\rangle|^2}.
\end{equation}
Then, for any $R,T>0$, the series in \eqref{G_R} is convergent and there exists a positive constant $C_j$ such that
\begin{equation}\label{estim-G_R}
    G_{R,j}(T)\leq e^{2C_j/T},\quad\forall\,0<T\leq 1.
\end{equation}
\end{prop}
\begin{proof}
For any $\lambda\geq0$, we set $f_1(\lambda)=e^{-2\lambda T+R\sqrt{\lambda}}$ and $f_2(\lambda)=e^{-\lambda T+R\sqrt{\lambda}}$. Then, the maximums of $f_1$ and $f_2$ is respectively attained for $\lambda=\left(\frac{R}{4T}\right)^2$ and $\lambda=\left(\frac{R}{2T}\right)^2$. From this remark and from assumption \eqref{ipB}, we deduce that
\begin{equation}\label{G_R-first-estim}
    \begin{split}
        G_{R,j}(T)&\leq \frac{R}{T^2}e^{R/T}\left(\frac{\lambda_1^pe^{R^2/(8T)}}{|\langle B\phi_j,\phi_1\rangle|^2}+\frac{e^{R^2/4T}}{b^2}\sum_{k=2}^\infty\lambda_k^{p+2q}e^{-\lambda_kT}\right)\\
        &\leq \frac{R}{T^2}e^{R/T}\left(\frac{\lambda_1^pe^{R^2/(8T)}}{|\langle B\phi_j,\phi_1\rangle|^2}+\frac{e^{R^2/4T}}{b^2}\sum_{k=1}^\infty\lambda_k^{p+2q}e^{-\lambda_kT}\right).
    \end{split}
\end{equation}
Now, for any $\lambda\geq0$, we set $g(\lambda)=\lambda^{p+2q}e^{-\lambda T}$. Then, we have
\begin{equation*}
    g(\lambda)\text{ is }\begin{cases}
    \text{increasing}&\text{if }0\leq \lambda< (p+2q)/T,\\
    \text{decreasing}&\text{if }\lambda\geq (p+2q)/T,
    \end{cases}
\end{equation*}
and attains its maximum for $\lambda=(p+2q)/T$. We define the following index:
\begin{equation*}
    k_1:=\sup\left\{k\in\N^*\,:\,\lambda_k\leq (p+2q)/T\right\}.
\end{equation*}
and we divide the sum in \eqref{G_R-first-estim} as follows
\begin{equation}\label{sum-first-estim}
    \sum_{k=1}^\infty \lambda_k^{p+2q}e^{-\lambda_kT}=\sum_{k\leq k_1-M}\lambda_k^{p+2q}e^{-\lambda_kT}+\sum_{k_1-M+1\leq k\leq k_1+M}\lambda_k^{p+2q}e^{-\lambda_kT}+\sum_{k\geq k_1+M+1}\lambda_k^{p+2q}e^{-\lambda_kT},
\end{equation}
where $M$ has been introduced in Lemma \ref{g13}. Of course, if $k_1\leq M$, we consider the first sum in \eqref{sum-first-estim} equal to $0$, and if $k_1\leq M-1$ the second sum starts from $k=1$.

For any $k\leq k_1-M$, from the properties of $g$, we have
\begin{equation*}
    \int_{\lambda_k}^{\lambda_{k+M}} \lambda^{p+2q}e^{-\lambda T}d\lambda\geq (\lambda_{k+M}-\lambda_k)\lambda_k^{p+2q}e^{-\lambda_k T}\geq \gamma \lambda_k^{p+2q}e^{-\lambda_k T},
\end{equation*}
and for any $k\geq k_1+M+1$, 
\begin{equation*}
    \int_{\lambda_{k-M}}^{\lambda_k} \lambda^{p+2q}e^{-\lambda T}d\lambda\geq (\lambda_{k}-\lambda_{k-M})\lambda_k^{p+2q}e^{-\lambda_k T}\geq \gamma \lambda_k^{p+2q}e^{-\lambda_k T}.
\end{equation*}
Therefore, using the estimates above in \eqref{sum-first-estim}, we deduce that
\begin{equation}\label{sum-second-estim}
    \sum_{k=1}^\infty \lambda_k^{p+2q}e^{-\lambda_kT}\leq \frac{2M}{\gamma}\int_0^\infty \lambda^{p+2q}e^{-\lambda T}d\lambda + \sum_{k_1-M+1\leq k\leq k_1+M}\lambda_k^{p+2q}e^{-\lambda_kT}.
\end{equation}
Moreover, since $g$ has a maximum at $\lambda=(p+2q)/T$, we have
\begin{equation}\label{sum-third-estim}
    k=k_1-M-1,k_1-M+2,\cdots,k_1+M\quad\Longrightarrow\quad \lambda_k^{p+2q}e^{-\lambda_kT} \leq ((p+2q)/T)^{p+2q}e^{-(p+2q)}.
\end{equation}
The integral term of \eqref{sum-second-estim} can be rewritten as
\begin{equation}\label{sum-fourth-estim}
    \int_0^\infty \lambda^{p+2q}e^{-\lambda T}d\lambda=\frac{1}{T}\int_0^\infty \left(\frac{s}{T}\right)^{p+2q}e^{-s}ds=\frac{1}{T^{1+p+2q}}\int_0^\infty s^{p+2q}e^{-s}ds=\frac{\Gamma(1+p+2q)}{T^{1+p+2q}},
\end{equation}
where $\Gamma(\cdot)$ is the Euler integral of second kind. Hence, from \eqref{sum-third-estim} and \eqref{sum-fourth-estim} we conclude that there exist two positive constants $C_{q,M},C_{q,M,\gamma}$ such that
\begin{equation*}
    \sum_{k=1}^\infty \lambda_k^{p+2q}e^{-\lambda_kT}\leq \frac{C_{q,M}}{T^{p+2q}}+\frac{C_{q,M,\gamma}}{T^{p+2q+1}}.
\end{equation*}
This last estimate implies the existence of a constant $C_j>0$ such that
\begin{equation*}
    G_{R,j}(T)\leq \frac{R}{T^2}e^{R/T}\left[\frac{\lambda_1^pe^{R^2/(8T)}}{|\langle B\phi_j,\phi_1\rangle|^2}+\frac{e^{R^2/4T}}{b^2}\left(\frac{C_{q,M}}{T^{p+2q}}+\frac{C_{q,M,\gamma}}{T^{p+2q+1}}\right)\right]\leq e^{2C_j/T},\quad\forall\, 0<T\leq 1.
\end{equation*}
\end{proof}

\begin{prop}\label{prop-jnullcontr}
Let $-\Delta$ be a Dirichlet-Neumann Laplacian on $X$ such that its eigenvalues satisfy \eqref{gap-required}. Let $B:X\to X$ be a linear bounded operator on $\Gi$ such that \eqref{ipB} holds.

Then, the pair $\{-\Delta,B\}$ is $j$-null controllable in any positive time $T>0$.
\end{prop}

\begin{proof}
Our aim is to show that, for all $T>0$, there exists a constant $K(T)>0$ such that for any initial condition $z_0\in X$ there exists a control $u\in L^2(0,T)$ which steers the solution $z$ of the linear system \eqref{lin-intro} (with $A=-\Delta$) to $0$ in a time $T$, that is, \eqref{lin-intro} is null controllable. To this purpose, we rewrite the null controllability condition by using the expansion in Fourier series of of the solution $z$ of \eqref{lin-intro}:
\begin{equation*}
    0=z(T;z_0,u)=\sum_{k\in\N^*}e^{-\lambda_k T}\la z_0,\phi_k\ra\phi_k -\int_0^T u(s)\sum_{k\in\N^*}e^{-\lambda_k(T-s)}\la B\phi_j,\phi_k\ra \phi_k\,ds.
\end{equation*}
Thanks to the othornormality of the family of eigenfunctions $\{\phi_k\}_{k\in\N^*}$, we obtain that the following relations have to be satisfied by $u$
\begin{equation*}
    \la z_0,\phi_k\ra=\int_0^T e^{\lambda_k s}u(s)\la B\phi_j,\phi_k\ra\,ds,\quad\forall\, k\in\N^*,
\end{equation*}
which can be rewritten in compact form as
\begin{equation}\label{mom-prob}
    \int_0^T e^{\lambda_k s}u(s)\,ds=d_k,\quad\forall\,k\in\N^*,
\end{equation}
where the coefficients $d_k:=\frac{\la z_0,\phi_k\ra}{\la B\phi_j,\phi_k\ra}$, $\forall\,k\in\N^*$, are well defined thanks to hypothesis \eqref{ipB}.

Since the eigenvalues of the Laplacian on $\Gi$ with $(\Di)$, $(\NN)$ or $(\NN\KK)$ satisfy the hypotheses of Proposition \ref{prop-bio-fam}, we use the biorthogonal family $\{\sigma_k\}_{k\in\N}$ to build our control, so that $u$ satisfies \eqref{mom-prob}:
\begin{equation*}
    u(t):=\sum_{k\in\N} d_k\sigma_k(t).
\end{equation*}
To conclude the proof, it remains to show that $u\in L^2(0,T)$. To this purpose, we use the bound \eqref{biort-bound} for the biorthogonal family $\{\sigma_k\}_{k\in\N}$:
\begin{equation*}
\begin{split}
    \norm{u}_{L^2(0,T)}&=\sum_{k\in\N^*}\left|\frac{\la z_0,\phi_k\ra}{\la B\phi_j,\phi_k\ra}\right|\norm{\sigma_k}_{L^2(0,T)}\leq \norm{z_0}\left(\sum_{k\in\N^*}\frac{\norm{\sigma_k}^2_{L^2(0,T)}}{|\la B\phi_j,\phi_k\ra|^2}\right)^{1/2}\\
    &\leq\norm{z_0}\left( Ce^{C/(T\gamma^2)}B(T,\gamma)\sum_{k\in\N^*}\left (1+ \frac{\gamma^2}{a_{k}(a_{k} + 2   \sqrt\lambda_1)}\right)^{2M}\frac{e^{-2\lambda_kT}e^{C\sqrt{\lambda_k}/\gamma}}{|\la B\phi_j,\phi_k\ra|^2}\right)^{1/2}.
\end{split}
\end{equation*}
We recall that $\lambda_k\asymp k^2$ and $a_k\asymp k^{-p}$ for $k\geq 2$. Thus, the right-hand side of the above inequality can be bounded by
\begin{equation*}
    \norm{z_0}\left(Ce^{C/(T\gamma^2)}B(T,\gamma)\sum_{k\in\N^*}\lambda_k^{2pM}\frac{e^{-2\lambda_kT+C\sqrt{\lambda_k}/\gamma}}{|\langle B\phi_j,\phi_k\rangle|^2}\right)^{1/2}.
\end{equation*}
From Proposition \ref{prop-G_R}, we know that the series in the formula above is convergent for any $T>0$. Hence $u\in L^2(0,T)$ and we have shown that there exists a constant $K(T)>0$ such that
\begin{equation*}
    \norm{u}_{L^2(0,T)}\leq K(T)\norm{z_0}.
\end{equation*}
\end{proof}
Now we can proceed with the proof of Theorem \ref{thm-ex-control-graph}.
\begin{proof}[Proof of Theorem \ref{thm-ex-control-graph}.]
From Proposition \ref{prop-jnullcontr} we deduce that the pair $\{-\Delta,B\}$ is $j$-null controllable in any time $T>0$. What remains to prove in order to use Theorem \ref{teo1} is to show that there exist $T_0,\nu>0$ such that
\begin{equation*}
    K(T)=\left(Ce^{C/(T\gamma^2)}B(T,\gamma)\sum_{k\in\N^*} \lambda_k^{2pM}\frac{e^{-2\lambda_kT}e^{C\sqrt{\lambda_k}/\gamma}}{|\la B\phi_j,\phi_k\ra|^2})\right)^{1/2}
\end{equation*}
satisfies \eqref{bound-control-cost}.

One can check that there exists a constant $R>0$ such that
\begin{equation*}
    K(T)\leq\left(Ce^{C/(T\gamma^2)}\frac{1}{T\gamma^2}\sum_{k\in\N^*}\lambda_k^{2pM}\frac{e^{-2\lambda_kT+C\sqrt{\lambda_k}/\gamma}}{|\langle B\phi_j,\phi_k\rangle|^2}\right)^{1/2}\leq \norm{z_0}G_{R,j}(T)^{1/2},\qquad\forall\,T\leq\min\left\{1,\frac{1}{\gamma^2}\right\},
\end{equation*}
where the function $G_{R,j}(\cdot)$ is defined in \eqref{G_R}. Thus, from estimate \eqref{estim-G_R}, we infer that
\begin{equation*}
    K(T)\leq e^{C_j/T},\qquad\forall\,0<T\leq T_0,
\end{equation*}
with $T_0:=\min\left\{1,\frac{1}{\gamma^2}\right\}$ and $C_j$ a suitable positive constant independent of $T$.

Therefore, we can apply Theorem \ref{teo1} and conclude that \eqref{bil-par-eq} is locally controllable to the $j$-th eigensolutions in any time $T>0$ by means of a control $u\in L^2(0,T)$ which satisfies \eqref{u-bound}.
\end{proof}

\subsection{Semi-global controllability results}
We have proved in Proposition \ref{prop-jnullcontr} that if the eigenvalues of the Dirichlet-Neumann Laplacian satisfy \eqref{gap-required} and the bounded linear operator $B$ fulfills the spreading property \eqref{ipB}, then the pair $\{-\Delta, B\}$ is $j$-null controllable. Furthermore, we have shown in the proof of Theorem \ref{thm-ex-control-graph} that, under the same assumptions on the operators $-\Delta$ and $B$, the control cost behaves as \eqref{bound-control-cost}.

Thus, it is possible to apply \cite[Theorem 1.4]{acue} and \cite[Thereom 1.5]{acue} and deduce the following semi-global controllability results.
\begin{teorema}\label{teo-glob1}
Let $-\Delta$ be a Dirichlet-Neumann Laplacian on $X=L^2(\Gi,\R)$ such that the eigenvalues $\{\lambda_k\}_{k\in\N^*}$ satisfy \eqref{gap-required}. Let $B:X\to X$ be a bounded linear operator on $\Gi$ such that \eqref{ipB} holds. 

Then, there exists $r_1>0$ such that for every $R>0$ there exists $T_R>0$ such that for all $y_0\in X$ with
\begin{equation}\label{strip}
    |\langle y_0,\phi_1\rangle-1|<r_1,\qquad \norm{y_0-\langle y_0,\phi_1\rangle\phi_1}\leq R,
\end{equation}
problem \eqref{bil-par-eq} is controllable to the first eigensolution $\varphi_1$ in time $T_R$.
\end{teorema}
The interpretation of Theorem \ref{teo-glob1} is the following: if the initial datum lies on a (infinite) strip, defined in \eqref{strip}, then we can steer the solution of \eqref{bil-par-eq} to the first eigensolution in time $T_R$ (see (a) in Figure \ref{fig-global} below). The idea to prove this semi-global result is to let the semigroup of the Laplacian evolve freely till the solution of \eqref{bil-par-eq} enters a suitable neighborhood of the first eigenfunction, and then to apply our local controllability result (see \cite[Theorem 1.4]{acue} for the proof).

Our second semi-global result  below, Theorem \ref{teo-glob2}, ensures the controllability in time $T_R$ of any initial condition $y_0\in X\setminus \phi_1^\perp$ to the evolution of its orthogonal projection along the first eigensolution $\varphi_1(t)$, that is the following trajectory
\begin{equation*}
    \zeta_1(t)=\langle y_0,\phi_1\rangle \varphi_1(t).
\end{equation*}
\begin{teorema}\label{teo-glob2}
Let $-\Delta$ be a Dirichlet-Neumann Laplacian on $X=L^2(\Gi,\R)$ such that the eigenvalues $\{\lambda_k\}_{k\in\N^*}$ satisfy \eqref{gap-required}. Let $B:X\to X$ be a bounded linear operator on $\Gi$ such that \eqref{ipB} holds. 

Then, for any $R>0$ there exists $T_R>0$ such that for all $y_0\in X$ with 
\begin{equation*}
    \norm{y_0- \langle y_0,\phi_1\rangle\phi_1}\leq R |\langle y_0,\phi_1\rangle|
\end{equation*}
problem \eqref{bil-par-eq} is controllable to the trajectory
\begin{equation*}
    \zeta_1(t)=\langle y_0,\phi_1\rangle\varphi_1(t)
\end{equation*}
in time $T_R$.
\end{teorema}
For the proof, see \cite[Theorem 1.5]{acue}.

\begin{figure}[h!]
\centering\subfloat[]{
\begin{tikzpicture}[scale=.7]
\fill[blue!30](4,-3)--(6,-3)--(6,3)--(4,3)--cycle;
\draw[] (0,0)  -- (10,0)node[below]{$\phi_1$};
\draw[] (2,4)node[left]{$\phi_1^\perp$} -- (2,-4);
\fill(5,0) node[below]{\footnotesize{$1$}} circle (.05);
\fill(3,0) node[below]{\footnotesize{$\varphi_1(T_R)$}} circle (.05);
\draw[] (4,-4) -- (4,4);
\draw[] (6,-4) -- (6,4);
\draw[] (2,3) node[left]{\footnotesize{$R$}} -- (9,3);
\draw[] (2,-3) node[left]{\footnotesize{$-R$}} -- (9,-3);
\draw[<->] (4,3.5) --node[above]{\footnotesize{$r_1$}} (6,3.5);
\fill(5.5,2) node[above]{\footnotesize{$y_0$}} circle (.05);
\draw[ultra thick,->] (5.5,2) -- (3,0);
\end{tikzpicture}
}\quad
\subfloat[]{
\begin{tikzpicture}[scale=.6]
\fill[blue!30](0,-4.5)--(7,0)--(0,4.5)--cycle;
\fill[blue!30](7,0)--(14,-4.5)--(14,4.5)--cycle;
\coordinate (v2) at (7,0);
\coordinate(v4) at (11.5,2);
\coordinate(v5) at (11.5,0);
\coordinate(v8) at (2,1);
\coordinate(v9) at (2,0);
%\tkzMarkAngle[size=1cm](v5,v2,v4);
%\tkzLabelAngle[pos=1.25](v5,v2,v4){\footnotesize{$\theta$}};
\draw (8,0) arc(0:30:.8);
\node at(8.4,0)[above]{\footnotesize{$\theta$}};
%\tkzMarkAngle[size=1.75cm](v8,v2,v9);
%\tkzLabelAngle[pos=2](v8,v2,v9){\footnotesize{$\hat{\theta}$}};
\draw (0,0)  -- (14,0)node[below]{$\phi_1$};
\draw (7,4.5)node[left]{$\phi_1^\perp$} -- (7,-4.5);
\draw (7,0) -- (13,3.86);
\draw[dashed] (13,3.86) -- (14,4.5);
\draw (7,0) -- (13,-3.86);
\draw[dashed] (13,-3.86) -- (14,-4.5);
\draw (7,0) -- (1,3.86);
\draw[dashed] (1,3.86) -- (0,4.5);
\draw (7,0) -- (1,-3.86);
\draw[dashed] (1,-3.86) -- (0,-4.5);
%\fill(11,0) node[below]{\footnotesize{$1$}} circle (.05);
%\fill(3,0) node[below]{\footnotesize{$\varphi_1$}} circle (.05);
\fill(10,0) node[below]{\footnotesize{$\zeta_1(T_R)$}} circle (.05);
%\fill(3.5,0) node[below]{\footnotesize{$\qquad\hat{\phi}_1(T_R)$}} circle (.05);
\fill(11.5,2) node[above]{\footnotesize{$y_0$}} circle (.05);
%\fill(2,1) node[above]{\footnotesize{$\hat{u}_0$}} circle (.05);
\draw (11.5,2) -- (11.5,0);
\draw (11.5,2) -- (7,0);
%\draw (2,1) -- (2,0);
%\draw (2,1) -- (7,0);
\fill(11.5,0) circle (.05);
%\fill(2,0) circle (.05);
\draw[ultra thick,->] (11.5,2) -- (10,0);
%\draw[ultra thick,->] (2,1) -- (3.5,0);
%\fill(13,-2.5) node{$Q_R$};
\end{tikzpicture}
}
\caption{Figure (a) illustrates the result of Thereom \ref{teo-glob1}: the solution of problem \eqref{bil-par-eq} with initial condition lying in the blue region can be driven to the first eigensolution $\varphi_1$ in time $T_R$ (which is uniform for any $y_0$ in the strip). In figure (b) we highlighted the cone of amplitude $2\arctan(R)$ of initial conditions which can be steered to the trajectory $\zeta_1$ in time $T_R$. Since $R$ is arbitrary, we are able to apply Theorem \ref{teo-glob2} for any $y_0\in X\setminus \phi_1^\perp$.  }\label{fig-global}
\end{figure}
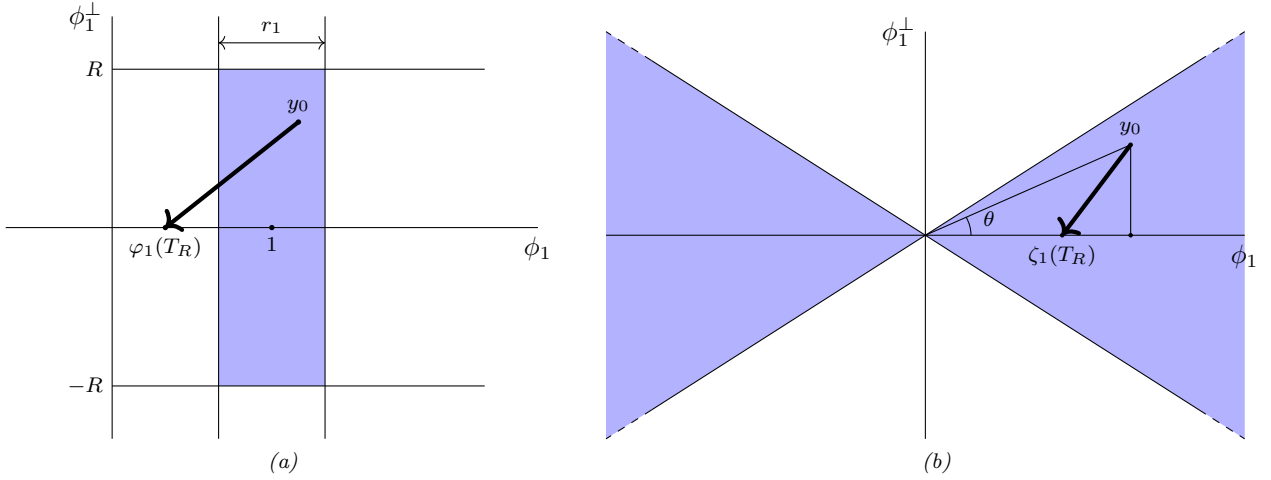
%\begin{assumptionI}
	%The bounded symmetric operator $B$ is such that there exists $C>0$ such that 
%$|\la\phi_j,B\phi_1\ra_{L^2}|\geq\frac{C}{j^{3}}$ for every 
%$j\in\N^*$.
%\end{assumptionI}

\section{Some explicit applications}
\subsection{Star graph}\label{stargraph}

\noindent
Let us consider a star-graph graph $\Gi$ composed by $3$ edges $\{e_1,e_2,e_3\}$ connected in one vertex $v$. We 
parametrize each $e_j$ 
with a coordinate going from $0$ to its length $L_j$ in $v$ (see Figure 
\ref{parametrizzazione}). 

\begin{figure}[H]
	\centering
	\includegraphics[width=\textwidth-100pt]{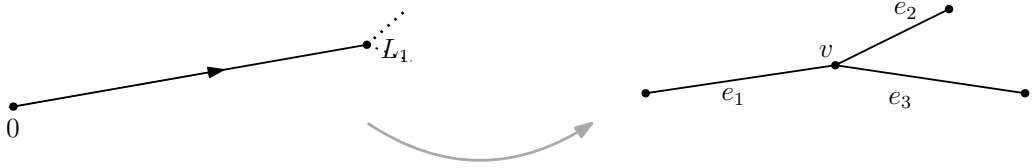}
	\caption{The figure shows the parametrization of a star graph with $3$ edges.}\label{parametrizzazione}
\end{figure}

In such a framework, $\psi\in L^2(\Gi,\R)$ defined on $\Gi$ is composed by three components 
$(\psi^1,\psi^2,\psi^3)$ where each $\psi^j\in L^2(e_j,\R)$. We assume that $-\Delta$ is a Neumann Laplacian and
\begin{align}\label{potential_introS}u\in 
L^2((0,T),\R),\ \ \ \ \ \  \ \ \ B:\psi=(\psi^1,\psi^2,\psi^3)\in L^2(\Gi,\R)\longmapsto 
\big(\mu(x)\psi^1,0,0\big) .\end{align} 
where $\mu$ is a sufficiently regular function. The bilinear system \eqref{main} becomes the problem
\begin{equation}\label{mainS}\begin{split}\begin{cases}
\dd_t\psi^1(t,x)-\dd_x^2\psi^1(t,x)+u(t)\mu(x)\psi^1(t,x)=0,\ \ \ \ \ \ \ \ &t\in(0,T),\ x\in (0,L_1),\\
\dd_t\psi^2(t,x)-\dd_x^2\psi^2(t,x)=0,\ \ \ \ \ \ \ \ &t\in(0,T),\ x\in (0,L_2),\\
\dd_t\psi^3(t,x)-\dd_x^2\psi^3(t,x)=0,\ \ \ \ \ \ \ \ &t\in(0,T),\ x\in (0,L_3),\\
\end{cases}\end{split}
\end{equation}
endowed with the following boundary conditions
\begin{equation}\label{mainS_boundaries}\begin{split}
\psi^1(L_1)=\psi^2(L_2)=&\ \psi^3(L_3),  \ \ \ \ \ \  \ \ \ \ \ \
\dd_x\psi^1(L_1)+\dd_x\psi^2(L_2)+\dd_x\psi^3(L_3)=0,\\
&\dd_x\psi^1(0)=\dd_x\psi^2(0)=\dd_x\psi^3(0)=0.\\
\end{split}
\end{equation}
We notice that the first eigenvalue $\lambda_1=0$ corresponds to the following eigenfunction $$\phi_1=\frac{1}{\sqrt{L_1+L_2+L_3}}(1,1,1).$$ 
The Neumann boundary conditions on the external vertices of $\Gi$ imply that each other $\phi_j$, with $j\geq 2$, satisfies $\phi_j^l(0)=0$ for every $l\leq 3$ which yields
\begin{align}\label{eq_eigenf}\phi_j(x)=\big(a^1_j\cos(x\sqrt{\lambda}_j),a_j^2\cos(x\sqrt{\lambda}_j),
a_j^3\cos(x\sqrt{\lambda} _j)\big), \end{align}
	with $\{a_j^l\}_{l\leq 3}\subset\R$. Now, the ($\NN\KK$) conditions in $V_i$ lead to
	 \begin{equation}\label{eq1}\begin{split}
	 a^1_j\cos(\sqrt{\lambda}_jL_1)=...=a^3_j\cos(\sqrt{\lambda}_jL_3),\ \ \ \ \ \ 
\sum_{l\leq 3} 
a^l_j\sin(\sqrt{\lambda}_jL_l)=0,
\end{split}\end{equation}
When, for instance, all the lengths of the edges of $\Gi$ are equal to $L$, it is easy to see that the sequence of non-repeated eigenvalues is obtained by reordering
\begin{align}\label{spettro_star_ugu}  \Big\{\frac{(2j-1)^2\pi^2}{4 L^2}\Big\}_{j\in\N^*},\ \  \ \ \  \ \ \ 
\Big\{\frac{j^2\pi^2}{L^2}\Big\}_{j\in\N},\end{align}
	which is $\big\{\frac{j^2\pi^2}{4 L^2}\big\}_{j\in\N}.$  Notice that the eigenvalues $\big\{\frac{(2j-1)^2\pi^2}{4 L^2}\big\}_{j\in\N^*}$ have multiplicity $N-1$ and then they are doubles in this specific case. When instead the 
lengths are such that each $L_j/L_k\not\in\Q$, the relations \eqref{eq1} yield the identities
 \begin{equation}\label{eq2}\begin{split}
	 \sum_{l\leq 3} \tan(\sqrt{\lambda}_jL_l)=0,\ \ \ \ \ \ \ \ \ \  \sum_{l\leq 
3}|a_j^l|^2{\sin(L_l\sqrt{\lambda}_j)\cos(L_l\sqrt{\lambda}_j)}=0.\\ 
	 \end{split}
	 \end{equation} 
As a consequence of \eqref{eq2}, the numbers $\sqrt\lambda_j$ with $j\geq 2$ are solutions of the 
transcendental equation
$$\sum_{l\leq 3} \tan(x L_l)=0.$$
The corresponding eigenfunctions are deduced by \eqref{eq_eigenf} with $\{a_j^l\}_{l\leq 3}\subset\R$. We choose 
these coefficients so that 
$\{\phi_j\}_{j\in\N^*}$ forms a Hilbert basis of $L^2(\Gi,\R)$. The orthonormality condition yields
\begin{align}\label{eq3}1=\sum_{l\leq 
3}\int_{0}^{L_l}(a_j^l)^2\cos^2(x\sqrt{\lambda}_j)dx=\sum_{l\leq 
3}(a_j^l)^2\Big(\frac{L_l}{2}+\frac{\cos(L_l\sqrt{\lambda}_j)\sin(L_l\sqrt{\lambda}_j)}{2\sqrt{\lambda}_j}
\Big).\end{align}
Now, we use \eqref{eq2} in \eqref{eq3} which ensures $1=\sum_{l=1}^3(a_j^l)^2{L_l}/{2}$. The 
continuity relation in the Neumann-Kichhoff conditions in $V_i$ implies 
$a_j^l=a_j^k\frac{\cos(\sqrt{\lambda}_j 
L_k)}{\cos(\sqrt{\lambda}_j L_l)}$ for $l\neq k$ and then
	\begin{equation}\label{pistola}
	\begin{split}
(a_j^k)^2\Bigg(L_k+\sum_{l\in\{1,2,3\}\setminus\{k\}}L_l\frac{\cos^2(\sqrt{\lambda}_j L_k)}{\cos^2(\sqrt{\lambda}_j L_l)}\Bigg)={2}.
	\end{split}
	\end{equation}
The relation \eqref{pistola} allows us to identify the parameters $\{a_j^l\}_{l\leq 3}\subset\R$ up to their signs as
	\begin{equation}\label{relation_coef}
	\begin{split}
	|a_j^k|=\sqrt\frac{2\prod_{m\neq k }\cos^2(\sqrt{\lambda}_j 
L_m)}{\sum_{l=1}^3 
L_l\prod_{m\neq l}\cos^2(\sqrt{\lambda}_j L_m)}.	\end{split}
	\end{equation}
\begin{teorema}\label{thm-control-star}
Let $\mu(x)=\cos\big(\frac{x \pi}{2L_1}\big)$ with $x\in (0,L_1).$ If $\{L_j\}_{j\leq 3}\in (\R^+)^N$ is so that
$\big\{1,L_1,...,L_3\big\}$ are linearly independent over $\Q$ and all the ratios $L_k/L_j$ are algebraic irrational numbers, then system \eqref{mainS} is locally controllable to the first eigensolution in any time $T>0$. Moreover, the control $u$ satisfies \eqref{u-bound}.
\end{teorema}
\begin{proof}
 The result is ensured by referring to Theorem \ref{thm-ex-control-graph}. First, the validity of the spectral gap follows by Lemma \ref{interessante} and Proposition \ref{gap_tadpole}. Second, by direct computations it is possible to deduce that
 $$\la B\phi_1,\phi_1\ra=\la \mu\phi_1^1,\phi_1^1\ra_{L^2(0,L_1)}=\frac{2L_1}{\pi(L_1+L_2+L_3)}\neq 0$$ and, for every $k\geq 2$,
\begin{equation}\label{degeneracy}\begin{split}
\la B\phi_1,\phi_k\ra&=\la \mu\phi_1^1,\phi_k^1\ra_{L^2(0,L_1)}=a_k^1\frac{ 2L_1\pi \cos(\sqrt\lambda_k L_1)}{\pi^2-4L_1^2\lambda_k}.\end{split}\end{equation}
 Thanks to \cite[Proposition 27]{_graphs} (arXiv version \cite[Proposition A.2]{_graphs}), for every $\epsilon>0$, there exists $C_1>0$ such that $|\cos(\sqrt{\lambda_k}L_l)|\geq {C_1}{\lambda_k}^\frac{-1-\epsilon}{2}$ for every $l\leq 3$. This identity and \eqref{relation_coef} ensure the existence of $C_2,C_3>0$ such that
\begin{equation}
	\label{oip}
	\begin{split}
	|a_k^1|&\geq \sqrt{\frac{2}{\sum_{l=1}^3L_l\cos^{-2}(\sqrt{\lambda_k} 
L_l)}}\geq\sqrt{\frac{2}{C_2{\lambda_k^{1+\epsilon}}\sum_{l=1}^3L_l}}\geq \frac{C_3}{{\lambda_k}^\frac{1+\epsilon}{2}},\ \ \ \ 
\ \ \  \ \forall\, k\geq 2.\\
	\end{split}
	\end{equation}
We use again $|\cos(\sqrt{\lambda_k}L_1)|\geq {C_1}{\lambda_k^\frac{-1-\epsilon}{2}}$ in \eqref{degeneracy}. Thanks to \eqref{oip}, there exists $C_4>0$ such that
\begin{equation*}\begin{split}
|\la B\phi_1,\phi_k\ra|&\geq \frac{C_4}{\lambda_k^{2+\epsilon}},\ \ \ \ \ \forall\, k\geq 2.\end{split}\end{equation*}
The result finally follows from Theorem \ref{thm-ex-control-graph}.
\end{proof}

\subsection{Tadpole graph}\label{tadpole}

Let $\Gi$ be a {tadpole graph} composed by two 
edges $\{e_1,e_2\}$ connected in an internal vertex $v$. The edge $e_1$ is self-closing and parametrized in the 
clockwise direction with a coordinate going from $0$ to $L_1$ (the length $e_1$). On the \virgolette{tail} $e_2$, 
we consider a coordinate going from $0$ in the to $L_2$ and we associate the $0$ to the external vertex $\tilde 
v$ (see Figure $\ref{parametrizzazione1}$).

\begin{figure}[H]
	\centering
	\includegraphics[width=\textwidth-60pt]{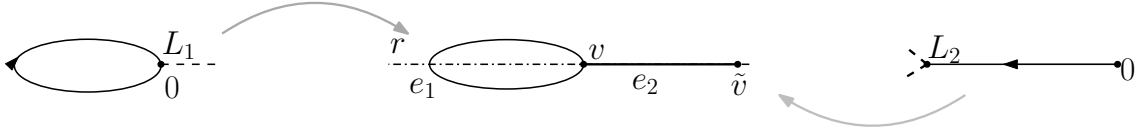}
	\caption{The parametrization of the tadpole graph and its symmetry axis $r$.}\label{parametrizzazione1}
\end{figure}

We consider the Hilbert space $L^2(\Gi,\R)$ composed by functions $\psi=(\psi^1,\psi^2)$ such that each 
$\psi^l\in L^2(e_l,\R)$. We assume that $-\Delta$ is a Neumann Laplacian and
\begin{align}\label{potential_introT}u\in 
L^2((0,T),\R),\ \ \ \ \ \  \ \ \ B:\psi=(\psi^1,\psi^2)\in L^2(\Gi,\R)\longmapsto 
\big(\mu(x)\psi^1,0\big),\end{align} 
where $\mu$ is a sufficiently regular function. The bilinear system \eqref{main} reads as
\begin{equation}\label{mainT}\begin{split}\begin{cases}
\dd_t\psi^1(t,x)-\dd_x^2\psi^1(t,x)+u(t)\mu(x)\psi^1(t,x)=0,\ \ \ \ \ \ \ \ &t\in(0,T),\ x\in (0,L_1),\\
\dd_t\psi^2(t,x)-\dd_x^2\psi^2(t,x)=0,\ \ \ \ \ \ \ \ &t\in(0,T),\ x\in (0,L_2),\\
\end{cases}\end{split}
\end{equation}
endowed with the following boundary conditions
\begin{equation}\label{mainT_boundaries}\begin{split}
\psi^1(L_1)=\psi^1(0)=&\,\psi^2(L_2),  \ \ \ \ \ \  \ 
\dd_x\psi^1(L_1)-\dd_x\psi^1(0)+\dd_x\psi^2(L_2)=0,\ \ \ \ \  \dd_x\psi^2(0)=0.\\
\end{split}
\end{equation}

Now, we refer to the proof of \cite[Theorem 15]{_graphs} (arXiv version \cite[Theorem 3.7]{_graphs}) and we notice that the sequence of eigenfunctions 
$\{\phi_k\}_{k\in\N^*}$ of $-\Delta$ can be assumed to be only composed by symmetric or skew-symmetric functions with 
respect to natural symmetry axis $r$ appearing in $\Gi$ (see Figure \ref{parametrizzazione1}).
If $\phi_k=(\phi_k^1,\phi_k^2)$ is skew-symmetric, then $$\phi^2_k\equiv 
0,\ \ \ \ \  \ \ \phi^1_k(0)=\phi^1_k(L_1/2)=\phi^1_k(L_1)= 0,\ \ \ \  \ \ \ \ 
\dd_x\phi^1_k(0)=\dd_x\phi^1_k(L_1).$$ As a consequence, the eigenvalues corresponding to the skew-symmetric 
eigenFunctions are $\left\{\frac{4 k^2\pi^2}{L_1^2}\right\}_{k\in \N^*}$ and such eigenfunctions are $$\left\{\Big(\sqrt{\frac{2}{L_1}}\sin\Big(x\frac{2k\pi}{L_1}\Big),0\Big)\right\}_{k\in\N^*}.$$
Afterwards, if $\phi_k=(\phi_k^1,\phi_k^2)$ is 
symmetric, then we have $$\dd_x\phi^1_k(L_1/2)=0,\ \ \ \ \ \phi_k^1(\cdot)=\phi_k^1(L_1-\cdot).$$ 
Firstly, we notice that the first eigenvalue $\lambda_1=0$ corresponds to a symmetric eigenfunction $$\phi_1=\frac{1}{\sqrt{L_1+L_2}}(1,1).$$ 
Secondly, the ($\NN$) 
condition
on $\widetilde v$ implies that each other symmetric eigenfunction associated to the eigenvalue $\lambda_k$ have the following formulation
$$\Big\{\Big(a_k^1\cos\Big(\sqrt{\lambda_k}\Big(x-\frac{L_1}{2}\Big)\Big),a_k^2\cos(\sqrt{\lambda_k} 
x)\Big)\Big\}_{k\in\N^*},\ \ \ \ \ \text{for}\ \ \ \ \  \{(a_k^1,a_k^2)\}_{k\in\N^*}\subset\R^2.$$
The $(\NN\KK)$ conditions in $v$ ensure that
$a_k^1\cos(\sqrt{\lambda_k}(L_1/2))=a_k^2\cos(\sqrt{\lambda_k}L_2))$ and 
$$2a_k^1\sin(\sqrt{\lambda_k}(L_1/2))+a_k^2\sin(\sqrt{\lambda_k}L_2))=0.$$
Similarly to what we noticed for the star graph, such eigenvalues can be made explicit when $L_1=L_2$.  
While if $L_1/L_2\not\in\Q$, then $\sqrt\lambda_k$ is a solution of the transcendental equation
\begin{equation}\label{transc_tad}2\tan(x(L_1/2))+\tan(xL_2))=0.\end{equation}
The boundary conditions also yield the following identities, for every $k\geq 2$,
 \begin{equation}\label{eq211}\begin{split}
	(a_k^1)^2\sin(L_1\sqrt{\lambda}_k)+(a_k^2)^2\sin(L_2\sqrt{\lambda}_k)\cos(L_2\sqrt{\lambda}_k)=0.\\ 
	 \end{split}
	 \end{equation} 
As for the star graph, we choose the coefficients $\{a_k^l\}_{l\leq 2}\subset\R$ so that 
$\{\phi_k\}_{k\in\N^*}$ forms a Hilbert basis of $L^2(\Gi,\R)$. The orthonormality condition yields that 
\begin{align}\label{eq311}1=(a_k^1)^2\Big(\frac{L_1}{2}+\frac{\sin(L_1\sqrt{\lambda}_k)}{2\sqrt{\lambda}_k}
\Big)+(a_k^2)^2\Big(\frac{L_2}{2}+\frac{\sin(L_2\sqrt{\lambda}_k)\cos(L_2\sqrt{\lambda}_k)}{2\sqrt{\lambda}_k}
\Big).\end{align}
Now, we use \eqref{eq211} in \eqref{eq311} which ensures $2=(a_k^1)^2{L_1}+(a_k^2)^2{L_2}$. As for the star graph, the 
continuity relation in the Neumann-Kichhoff conditions in $v$ allows us to identify the parameters $\{a_k^l\}_{l\leq 2}\subset\R$, up to their sign, as
\begin{equation}\label{coeff_tad}|a_k^1|=\sqrt\frac{2\cos^2(\sqrt{\lambda}_k 
L_2)}{
L_1\cos^2(\sqrt{\lambda}_k L_m)+L_2\cos^2(\sqrt{\lambda}_k L_1/2)},\ \ \ \ \ \ \ \ \  \ |a_k^1|=\sqrt\frac{2\cos^2(\sqrt{\lambda}_k 
L_1/2)}{L_1\cos^2(\sqrt{\lambda}_k L_m)+L_2\cos^2(\sqrt{\lambda}_k L_1/2)}.\end{equation}
\begin{teorema}\label{thm-control-tad}
Let $\mu(x)=x$ with $x\in (0,L_1).$ If $\{L_1,L_2\}\in (\R^+)^2$ is so that
$\big\{1,L_1,L_2\big\}$ are linearly independent over $\Q$ and the ratio $L_1/L_2$ is an algebraic irrational number, then system \eqref{mainT} is locally controllable to the first eigensolution in any positive time $T>0$. Moreover, the control $u$ satisfies \eqref{u-bound}.
\end{teorema}
\begin{proof}
 The proof follows the one of Theorem \ref{thm-control-star}. First, the spectral gap is valid thanks to Lemma \ref{interessante} and Proposition \ref{gap_tadpole}. Second, we can compute
 $$\la B\phi_1,\phi_1\ra=\la \mu\phi_1^1,\phi_1^1\ra_{L^2(0,L_1)}=\frac{L_1}{2(L_1+L_2)}\neq 0.$$ We consider a symmetric eigenfunction $\phi_k$ and
\begin{equation}\label{degeneracy1}\begin{split}
\la B\phi_1,\phi_k\ra&=\la \mu\phi_1^1,\phi_k^1\ra_{L^2(0,L_1)}=a_k^1\frac{ L_1\sin(\sqrt\lambda_k L_1/2)}{\sqrt\lambda_k}.\end{split}\end{equation}
Now, from \eqref{transc_tad}, we have the following relation
\begin{equation*}2\cos(\sqrt{\lambda_k}L_2)\sin(\sqrt{\lambda_k}L_1/2)+\cos(\sqrt{\lambda_k}L_1/2)\sin(\sqrt{\lambda_k}L_2)=0.\end{equation*}
We consider \cite[Remark 28]{_graphs} (arXiv version \cite[Proposition A.3]{_graphs}) with $\{L_1/2,L_2\}$ and the previous relation yields that, for every $\epsilon>0$, there exists $C_1>0$ such that 
\begin{equation}\label{lowerbound}\begin{split}|\sin(\sqrt{\lambda_k}L_1/2)|\geq \frac{C_1}{\sqrt{\lambda_k}^{1+\epsilon}},\ \ \ \ \   
|\sin(\sqrt{\lambda_k}L_2)|\geq \frac{C_1}{\sqrt{\lambda_k}^{1+\epsilon}},\\
|\cos(\sqrt{\lambda_k}L_1/2)|\geq \frac{C_1}{\sqrt{\lambda_k}^{1+\epsilon}},\ \ \ \ \   
|\cos(\sqrt{\lambda_k}L_2)|\geq \frac{C_1}{\sqrt{\lambda_k}^{1+\epsilon}}.\end{split}\end{equation}
If we use the second line of \eqref{lowerbound} in \eqref{coeff_tad}, then we deduce the existence of $C_2>0$ such that
\begin{equation*}
	\begin{split}
	|a_k^1|&\geq \frac{C_2}{\sqrt{\lambda_k}^{1+\epsilon}}.\\
	\end{split}
	\end{equation*}
We use now the first line of \eqref{lowerbound} in \eqref{degeneracy1}. The previous relation implies that, for every $\epsilon>0$, there exists $C_3>0$ such that,
\begin{equation*}\begin{split}
|\la B\phi_1,\phi_k\ra|&\geq \frac{C_3}{\sqrt{\lambda_k}^{3+2\epsilon}}.\end{split}\end{equation*}
When we consider $\phi_k$ as a skew-symmetric eigenfunction, we have 
\begin{equation*}\begin{split}|\la B\phi_1,\phi_k\ra|&\geq \frac{L_1}{j\pi}.\end{split}\end{equation*}
The result finally follows from Theorem \ref{thm-ex-control-graph}.
\end{proof}

\subsection{Space filtering in presence of multiple eigenvalues}
In this section, we discuss how to deal with systems with multiple eigenvalues for which the spectral gap \eqref{gap-required} is not satisfied. Examples of this situation are the star graphs and the tadpoles with some edges of same length. For instance, in the specific case of a star graph with $3$ edges of equal length, the spectrum is given by \eqref{spettro_star_ugu}, where multiples eigenvalues appear.

We apply the technique of filtation of the Hilbert space $L^2(\Gi,\R)$ in order to obtain simple spectrum. The main idea is to investigate the existence of a subspace $\Hi$, preserved by the dynamics of \eqref{main}, where the spectrum of the Laplacian turns out to be simple. In such case, controllability can be studied in $\Hi$ under suitable assumptions on $B$. This idea was already adopted in \cite[Section 6]{_graphs1} and \cite{_ammari, _ammari1} for the bilinear controllability of the Schr\"odinger equation. The application to our framework is straightforwards and, here, we just present the main steps in an explicit example.

Let $\Gi$ be a star graph with four edges parametrized as in Section \ref{stargraph}. We assume $L_1=L_2=L_3=1$ and $L_4=L>0$. We consider a Dirichlet Laplacian $-\Delta$ on $\Gi$ and its spectrum contains the eigenvalues
$$\tilde\lambda_k={k^2\pi^2},\ \ \ \  \ \ \forall\, k\in\N^*.$$
We notice that $\{\tilde\lambda_k\}_{k\in\N^*}$ are multiple eigenvalues (if $L\in\R\setminus\Q$, then they are actually double eigenvalues). Two families of corresponding orthonormalized eigenfunctions are $\{f_k\}_{k\in\N^*}$ and $\{g_k\}_{k\in\N^*}$ such that
\needspace{5pt
}$$f_k=(\sin(k\pi x),-\sin(k\pi x),0,0),\ \ \ \  \ \ \forall\, k\in\N^*,$$
$$g_k=\Bigg(\sqrt\frac{1}{2}\sin(k\pi x),\sqrt\frac{1}{2}\sin(k\pi x),-\sin(k\pi x),0\Bigg),\ \ \ \  \ \ \forall\, k\in\N^*.$$
Now, we introduce the space $\Hi=\overline{\spn\{f_k\ :\ k\in\N^*\}}^{\ L^2}$ and we notice that
$$\Hi=\big\{(\tilde\psi,-\tilde\psi,0,0)\ :\ \tilde\psi\in L^2((0,1),\R)\big\}.$$
The spectrum of the Dirichlet Laplacian $-\Delta$ in $\Hi$ is composed by the eigenvalues $\{\tilde\lambda_k\}_{k\in\N^*}$ which are now simple. Thus, Theorem \ref{thm-ex-control-graph} can be applied in this space $\Hi$ for a suitable operator $B$ defined as follows.

\begin{itemize}
    \item We consider $B$ so that, for any initial data in $\Hi$, the corresponding solution of \eqref{main} remains in $\Hi$ for any time. A possible choice of $B$ is
$$B:\psi=(\psi^1,\psi^2,\psi^3,\psi^4)\in L^2(\Gi,\R)\longmapsto B\psi=(\mu \psi^1,\mu\psi^2,\mu_1\psi^3,\mu_2\psi^4),$$
where $\mu,\mu_1$ and $\mu_2$ are bounded functions. Indeed, 
$$\text{for any }\ \psi=(\tilde\psi,-\tilde\psi,0,0)\in\Hi,\ \text{ we have }\ B\psi=(\mu\tilde\psi,-\mu\tilde\psi,0,0)\in\Hi.$$

\item We choose $B$ so that, conditions \eqref{ipB} are verified with respect to $\{f_k\}_{k\in\N^*}.$ In our case, one can identify $\Hi$ with $L^2(0,1)$ and $B$ with the multiplication operator by a function $\mu$ in $L^2(0,1)$. This is possible because each element in $\Hi$ is uniquely determined by its first component which belongs to $L^2(e_1)\equiv L^2(0,1)$. So, our problem boils down to the controllability of a one-dimensional control system in the interval $(0,1)$ (see the examples in \cite{acue, acu, cu}). A possible choice of the function $\mu$ is $\mu(x)=x^2$ as explained in \cite[Section 6.1]{acue}. Finally, the operator $B$ defined by
$$B\psi=(x^2 \psi^1,x^2\psi^2,\mu_1\psi^3,\mu_2\psi^4),$$
with $\mu_1$, $\mu_2$ any bounded functions, allows to locally control our original problem in $\Hi$ to any eigensolution $$(e^{-j^2\pi^2 t}\sin(j\pi x),-e^{-j^2\pi^2 t}\sin(j\pi x),0,0).$$

\end{itemize}

This strategy is not only valid on compact graphs in presence multiple eigenvalues, but also on infinite graphs. In such a framework, it is possible to construct eigenfunctions for the Laplacian in presence of suitable substructures of the graph. It is true even though there are edges of infinite length and the Laplacian has not compact resolvent. In the span of these eigenfunctons, one can study the controllability of the system as explained above. For further details, we refer to \cite{_ammari,_ammari1} where this idea is applied for the Schr\"odinger equation.

\section*{Acknowledgements}
This work was partly supported by the National Group for Mathematical Analysis, Probability and Applications (GNAMPA) of the Italian Istituto Nazionale di Alta Matematica ``Francesco Severi''; moreover, the first and third author acknowledge support by the Excellence Department Project awarded to the Department of Mathematics, University of Rome Tor Vergata, CUP E83C18000100006.

The second author was financially supported by the Agence nationale de la recherche of the French government through the grant {\it ISDEEC} (ANR-16-CE40-0013).

\appendix

\section{Appendix: Proof of Proposition \ref{gap_tadpole}}\label{proof_gap_star} 
In the current appendix, we use the techniques developed in \cite[Section\ 3]{_graphs1} in order to prove the validity of the spectral gap of Proposition \ref{gap_tadpole}.

As in \cite[Section\ 3]{_graphs1}, we start by considering an ordered sequence of pairwise distinct numbers ${\bf {\upnu}}=(\nu_k)_{k\in\Z^*}\subset\R$ such that 
there exist $M\in\N^*$ and $\delta>0$ such that 
\begin{equation}\begin{split}\label{gapp11}
\inf_{\{k\in\Z^*\ :\ k+M\neq 0 \}}|{\nu_{k+M}}-{\nu_k}|\geq\delta M.\\
\end{split}\end{equation}\needspace{2mm}
\noindent
This property implies that it does not exist $M$ consecutive indices $k,k+1\in\Z^*$ such that 
$|\nu_{k+1}  -  \nu_{k}| < \delta$ and then, there exist some $j\in\Z^*\setminus\{-1\}$ such that 
$|\nu_{j+1}-\nu_{j}|\geq\delta$. This distribution of the numbers ${\bf {\upnu}}$ allows us to define a partition of $\Z^*$ in subsets $\{E_m\}_{m\in\Z^*}$ (see \cite[Section 3.3]{_graphs1} for further details on this partition). The partition of $\Z^*$ in subsets $\{E_m\}_{m\in\Z^*}$ also defines an equivalence 
relation in $\Z^*$ and $\{E_m\}_{m\in\Z^*}$ are the equivalence classes corresponding to such relation. Clearly, $|E_m|\leq M$ thanks to 
\eqref{gapp11}. Let $s(m)$ 
be the smallest element of $E_m$. For every ${\bf {x}}:=\{x_k\}_{k\in\Z^*}\subset\R$ and $m\in\Z^*,$ we define ${\bf x}^m$ as the vector in $\R^{|E_m|}$ composed by those elements of ${\bf x}$ with indices in $E_m$, {\it i.e.}
$${\bf 
{x}}^m:=\{x^m_l\}_{l\leq |E_m|},\ \ \ \ \  : \ \ \  \ \ x_l^m=x_{s(m)+(l-1)},\ \ \ \ \ \forall\, l\leq |E_m|.$$
We are finally ready to introduce the matrix $F_m({\bf {\upnu}}^m):\R^{|E_m|}\rightarrow \R^{|E_m|}$ with components
\begin{equation*}
\begin{split}
F_{m;j,k}({\bf {\upnu}}^m):=\begin{cases}
\prod_{\underset{ l\leq k}{l\neq j}}(\nu^m_j-\nu_l^m)^{-1},\ \ \ \ \ & j\leq k,\\
1,\ \ \ \ \ \ \ \ \ \ \ \ \ \ \ \ & j=k=1,\\
0,\ \ \ \ \ \ \ \ \ \ \ \ \ \ \ \ & j>k,\\
\end{cases}\ \ \ \ \ \ \ \ \ \ \ \ \ \ \forall\, j,k\leq {|E_m|}.
\end{split}
\end{equation*}For each $k\in\Z^*$, there exists $m(k)\in\Z^*$ such that $k\in E_{m(k)}$, while $s(m(k))$ 
represents the smallest element of $E_{m(k)}$. In the following Lemma, we state \cite[Lemma\ 3.16]{_graphs1} which characterizes the matrix $
F_{m;j,k}({\bf {\upnu}}^m)$.

\begin{lemma}{\cite[Lemma\ 3.16]{_graphs1}}\label{entire1}
Let ${\bf {\upnu}}:=(\nu_k)_{k\in\Z^*}$ be an ordered sequence of pairwise distinct real numbers satisfying 
$(\ref{gapp11})$. Let $G$ be an entire function such that $G\in L^\infty(\R,\R)$. Let $J,I>0$ be such 
that 
$$|G(z)|\leq J e^{I|z|},\ \ \ \ \  \ \forall\, z\in\C.$$ If $({\nu_k})_{k\in\Z^*}$ are simple zeros of $G$ such that there 
exist 
$\tilde d\geq 0$, $C>0$ such that 
$$|G'(\nu_k)|\geq \frac{C}{|k|^{1+\tilde d}},\ \ \ \ \ \forall\, k\in\Z^*,$$ then there exists $C>0$ so that 
$$Tr\Big(F_{m}({\bf \upnu}^m)^*F_{m}({\bf \upnu}^m)\Big)\leq C\min\{|l|\in E_m\}^{2(1+{\tilde d})},\ \ \ \ \ \ 
\forall\, m\in\Z^*.$$
\end{lemma}

As a consequence of the previous Lemma, we have the following one.

\begin{lemma}\label{entire2}
Let ${\bf {\upnu}}:=(\nu_k)_{k\in\Z^*}$ be an ordered sequence of pairwise distinct real numbers satisfying 
$(\ref{gapp11})$. Let $G$ be an entire function such that $G\in L^\infty(\R,\R)$. Let $J,I>0$ be such 
that 
$$|G(z)|\leq J e^{I|z|},\ \ \ \ \ \ \forall\, z\in\C.$$ If $({\nu_k})_{k\in\Z^*}$ are simple zeros of $G$ such that there 
exist 
$\tilde d\geq 0$, $C>0$ such that 
$$|G'(\nu_k)|\geq \frac{C}{|k|^{1+\tilde d}},\ \ \ \ \  \ \forall\, k\in\Z^*,$$ then, for every $\epsilon>0$, there exists $C>0$ such that
\begin{equation*}|\nu_{k+1}-\nu_k|\geq C  |k|^{-1-\epsilon},\ \ \ \ \ \ \ \ \forall\, k\in\Z^*\setminus\{-1\}\end{equation*}
\end{lemma}
\begin{proof}
As introduced before, from $(\ref{gapp11})$, it does not exist $M$ consecutive $k,k+1\in\Z^*$ such that $|\nu_{k+1}  -  \nu_{k}| < \delta$. It is clear that if $k\in E_{m(k)}$ and $k+1\in E_{m(k)+1}$, then $E_{m(k+1)}=E_{m(k)+1}$ and $|\nu_{k+1}  -  \nu_{k}| > \delta$. Let us consider now the case such that $k,k+1\in E_{m(k)}$ which yields $E_{m(k)}=E_{m(k+1)}$. First, we notice that, for every $l,n\in E_{m(k)}$, we have \begin{equation}\label{eq_12}|\nu_{l}  -  \nu_{n}| <g:= (M-1)\delta,\end{equation}
where $g$ clearly does not depend on $k$. Second, thanks to Lemma \ref{entire1}, there exist $C_1,C_2>0$ so that 
\begin{equation}\begin{split}\label{eq_11}\prod_{ \underset{l< k+1}{l\in  E_{m(k)}}}(\nu_{k+1}-\nu_l)^{-2}&\leq 1+ \sum_{n\in E_{m(k)}\setminus\{s(m(k))\}}\prod_{ \underset{l< n}{l\in E_{m(k)}}}(\nu_{k+1}-\nu_l)^{-2}=Tr\Big(F_{m(k)}({\bf \upnu}^{m(k)})^*F_{m(k)}({\bf \upnu}^{m(k)})\Big)\\
&\leq C_1\min\{|l|\in E_{m(k)}\}^{2(1+{\tilde d})}\leq C_1 (|k|-M+1)^{2(1+{\tilde d})}\leq C_2 |k|^{2(1+{\tilde d})}.\end{split}\end{equation}
Third, due to \eqref{eq_12}, for $C_3:=\min\{1,g^{-M+2}\}$, it yields that
$$(\nu_{k+1}-\nu_k)^2=(\nu_{k+1}-\nu_k)^2\frac{\prod_{ \underset{l< k+1}{l\in E_{m(k)}\setminus\{k\}}}(\nu_{k+1}-\nu_l)^{2}}{\prod_{ \underset{l< k+1}{l\in E_{m(k)}\setminus\{k\}}}(\nu_{k+1}-\nu_l)^{2}}\geq C_3 \prod_{ \underset{l< k+1}{l\in E_{m(k)}}}(\nu_{k+1}-\nu_l)^{2}.$$
The statement follows by gathering the last relation with the identity \eqref{eq_11}.
\end{proof}
We are finally ready to prove Proposition \ref{gap_tadpole}. 
\begin{proof}[Proof of Proposition \ref{gap_tadpole}]
The statement for the cases where $\Gi$ is either a tadpole, a 
two-tails tadpole or a double-rings graph is provided as \cite[Proposition 8]{_graphs} (arXiv version \cite[Proposition 2.5]{_graphs}). We now prove the result when $\Gi$ is a $N$ edges star graph. We assume $\lambda_1\neq 0$. When $\lambda_1=0$, it enough to separate the first eigenvalue from the rest of the sequence. 
Let us denote by ${\bf {\upnu}}:=(\nu_k)_{k\in\Z^*}$ the sequence with elements
$$\nu_k=\sqrt{\lambda_k},\ \ \ \ \ \forall\, k>0;\ \ 
\ 
\ \  \ \nu_k=-\sqrt{\lambda_{-k}},\ \ \ \ \ \forall\, k< 0.$$
Let $I_1\subseteq \{1,...,N\}$ be the set of indices of those edges containing an 
external vertex equipped with Neumann boundary condition and $I_2:=\{1,..,N\}\setminus I_1$. We introduce the entire map
	\begin{equation*}\begin{split}
	G(x):=&\prod_{l\in I_2}\sin({x} L_l)\prod_{l\in I_1}\cos({x} L_l)\Big(\sum_{l\in I_2} \cot({x} L_l)+\sum_{l\in 
I_1} \tan({x} L_l)\Big). 
	\end{split}\end{equation*}
The map $G(x)$ verifies the hypotheses of Lemma \ref{entire2}, as explained in the proof of \cite[Theorem 5.1]{_graphs1}, which gives the claim.
\end{proof}

\section{Appendix: Proof of Proposition \ref{prop-bio-fam}}\label{bioApp}
The aim of this appendix is to adapt to the current setting the techniques developed in the works \cite{Bio2gap} and \cite{bio_primo} in order to prove Proposition \ref{prop-bio-fam}. The mentioned paper considers the biorthogonal family to a family of exponentials $\{e^{\lambda_kt}\}_{k\in\N^*}$ where $\{\lambda_k\}_{k\in\N^*}$ is a sequence of ordered positive real numbers such that there exist $0<\gamma<\gamma^*$ verifying, for every $k\in\N^*$,
\begin{equation*}
\begin{split}
\begin{cases}
\sqrt{\lambda_{2k}}-\sqrt{\lambda_{2k -1 }}\geq \gamma,\\\\
\sqrt{\lambda_{2k+1}}-\sqrt{\lambda_{2k }}\geq \gamma^*.
\end{cases}
\end{split}
\end{equation*}
The authors provide an upper bound for the $L^2-$norm of each element of such biorthogonal family.

Let $\lfloor r\rfloor$ be the entire part of $r\in\R$. We introduce the notation, for $\rho>0$,
$$ N_k(\rho):=card\{m\in\N^*\ :\ 0<|\lambda_m-\lambda_k|\leq \rho\}.$$

\begin{lemma}\label{biort1}
Let $\{\lambda_k\}_{k\in\N^*}$ be a sequence of ordered non-negative real numbers.
Assume that there exist $M\in\N^*$, $\gamma>0$ and an order sequence of decreasing positive real numbers $\{a_k\}_{k\in\N^*}$ verifying, for every $k\in\N^*$,
\begin{equation}\label{HP1}
\begin{split}
\begin{cases}
\sqrt{\lambda_{k+M}}-\sqrt{\lambda_{k }}\geq \gamma,\\\\
\sqrt{\lambda_{k+1}}-\sqrt{\lambda_{k }}\geq a_k.
\end{cases}
\end{split}
\end{equation} Then, for all $k\in\N^*$, we have $N_k(\rho)=0$ for $\rho \in \big(0, a_{k}(a_{k} + 2   \sqrt\lambda_1)\big)$. In addition, for every $k\in\N^*$ and $\rho>0$
$$N_m(\rho)\leq  2\left(\frac{\sqrt\rho}{\gamma}+1\right) M.$$
\end{lemma}
\begin{proof}
The first relation is proved as follows. For $k\in\N^*\setminus\{1\},$ we notice
$$\lambda_{k}-\lambda_{k-1}=\left(\sqrt{\lambda_{k}}-\sqrt{\lambda_{k-1}}\right)\left(\sqrt{\lambda_{k}}+\sqrt{\lambda_{k-1}}\right)\geq a_{k-1}\left(a_{k-1}+2\sqrt{\lambda_{k-1}}\right)\geq a_{k-1}\left(a_{k-1}+2\sqrt{\lambda_1}\right)$$ A similar inequality holds when we consider the gap $\lambda_{k+1}-\lambda_{k}$. Now, the closest element to $\lambda_k$ belonging to $\{\lambda_k\}_{k\in\N^*}$ is either $\lambda_{k+1}$, or $\lambda_{k-1}$, which yields the first relation. The estimation for $N_1(\cdot)$ follows equivalently.

The last relation is a consequence of the first hypothesis in \eqref{HP1}. First, we notice that for every $k,m\in\N^*$ such that $\left|\lambda_{k}-\lambda_{m}\right|\leq \rho$, we have 
\begin{align}\label{gap1}\left|\sqrt{\lambda_{k}}-\sqrt{\lambda_{m}}\right|\leq \sqrt{\left|\lambda_{k}-\lambda_{m}\right|}\leq \sqrt\rho.\end{align}
Second, we assume $|k-m|>M$ and \eqref{HP1} yields
$$\left|\sqrt{\lambda_{k}}-\sqrt{\lambda_{m}}\right|\geq \left\lfloor\frac{|k-m|}{M}\right\rfloor \gamma \geq\left(\frac{|k-m|}{M}-1\right) \gamma.$$
Thanks to the last relation and \eqref{gap1}, we obtain
\begin{align}\label{gap2}|k-m|\leq M\left(\frac{\left|\sqrt\lambda_{k}-\sqrt\lambda_{m}\right|}{\gamma}+1\right)\leq  \left(\frac{\sqrt\rho}{\gamma}+1\right) M.\end{align}
Obviously, such relation is also valid when $|k-m|\leq M $.
Let $k_*$ be the smallest natural number such that $\lambda_{k}-\lambda_{k_*}\leq \rho$. Let $k^*$ be the largest natural number such that $\lambda_{k^*}-\lambda_{k}\leq \rho$. We conclude the proof by observing that, thanks to \eqref{gap2}, we have 
$$N_k(\rho)= k^*-k_*= k^*-k + k-k_* \leq  2\left(\frac{\sqrt\rho}{\gamma}+1\right) M.$$
\end{proof}

In what follows, we assume that $\{\lambda_k\}_{k\in\N^*}$ is a sequence of ordered positive real numbers such that $\sum_{k\in\N^*} \lambda_k^{-1}<+\infty $. Such choice allows us to define the following sequence of functions 
$$F_k(z):=\prod_{m\in\N^*\setminus\{k\}}\left(1-\frac{i z- \lambda_k}{\lambda_m-\lambda_k}\right),\quad \forall\, k\in\N^*.$$
We refer to \cite[Section\ 3.3]{Bio2gap} for further details on their well-posedness .

\begin{lemma}\label{biort2}
Let $\{\lambda_k\}_{k\in\N^*}$ be a sequence of ordered non-negative real numbers such that $\sum_{k\in\N^*} \lambda_k^{-1}<+\infty $. Assume that the hypotheses of Lemma \ref{biort1} are verified.
For every $k\in\N^*$, we have the following estimate
$$\ln \left|F_k\left(z-i\lambda_k\right)\right| \leq\int_0^{+\infty} N_k(\rho)\frac{|z|}{\rho^2 +\rho |k|} d\rho.$$\end{lemma}
\begin{proof}
The result is proved exactly as \cite[Lemma 3.3]{Bio2gap}.
\end{proof}
\begin{lemma}\label{biort3}
Assume that the hypotheses of Lemma \ref{biort2} are verified.
For every $k\in\N^*$, we have the following estimate
$$\left|F_k(z)\right|\leq \left (1+ \frac{|z|+\lambda_k}{a_{k}\left(a_{k} + 2   \sqrt\lambda_1\right)}\right)^M e^{\frac{2M\pi }{\gamma}\sqrt{|z|+\lambda_k}}.$$
\end{lemma}
\begin{proof}
The proof follows by adapting to our context the techniques developed in the proof of \cite[Lemma 3.2]{Bio2gap}. Let  us focus on the case $k\in\N^*\setminus\{1\}$. The proof for $k=1$ follows from the same reasons by using the specific relation from Lemma \ref{biort1}.  First, thanks to Lemma \ref{biort1} and Lemma \ref{biort2}, we have $$\ln|F_k(z-i \lambda_k)|\leq \frac{2M}{\gamma}\int_0^{+\infty}\frac{\sqrt{\rho}|z|}{\rho^2 +\rho |k|}d\rho+ M\int_{a_{k}\left(a_{k} + 2   \sqrt\lambda_1\right)}^{+\infty}\frac{|z|}{\rho^2 +\rho |k|}d\rho,\quad\forall\,k\in\N^*\setminus\{1\}.
$$
As it has been proved in the aforementioned work, the first integral appearing in the previous relations is bounded by $\pi\sqrt{|z|}$, while the second one by $\ln\left (1+ \frac{|z|}{a_{k}(a_{k} + 2   \sqrt\lambda_1)}\right)$. Therefore, we have 
$$\ln|F_k(z-i \lambda_k)|\leq \frac{2M\pi }{\gamma}\sqrt{|z|}+ M\ln\left (1+ \frac{|z|}{a_{k}\left(a_{k} + 2   \sqrt\lambda_1\right)}\right),\quad\forall\,k\in\N^*\setminus\{1\},$$
which implies the result by acting the change of variable $\tilde z= z-i\lambda_k $ in the following resulting inequality
$$|F_k(z-i \lambda_k)|\leq \left (1+ \frac{|z|}{a_{k}\left(a_{k} + 2   \sqrt\lambda_1\right)}\right)^M e^{\frac{2M\pi }{\gamma}\sqrt{|z|}},\quad\forall\,k\in\N^*\setminus\{1\}.$$
\end{proof}

\begin{lemma}\label{biort4}
Assume that the hypotheses of Lemma \ref{biort2} are verified.
There exists $C>0$ independent of $k\in\N^*$, $\{a_k\}_{k\in\N^*},$ $\gamma>0$ and $z\in \C$ such that
$$|F_k(z)|\leq C\left(1+ \frac{\gamma^2}{a_{k}\left(a_{k} + 2   \sqrt\lambda_1\right)}\right)^M e^{\frac{4M\pi }{\gamma}\sqrt{|z|+\lambda_k}},\quad\forall\,k\in\N^*.$$
\end{lemma}
\begin{proof}
The above property follows from Lemma \ref{biort3} and from the inequality  $\frac{|z|+\lambda}{\gamma^2}\leq e^\frac{\sqrt{|z|+\lambda}}{\gamma}.$
\end{proof}

The remaining part of the proof of Proposition \ref{prop-bio-fam} is the same of the corresponding one of \cite[Theorem 2.3]{Bio2gap} and \cite[Theorem 2.4]{bio_primo}. Here, we just underline the main strategy and we refer to \cite[Section 4.4 and Section 4.5]{bio_primo} for further details. The core of this part is to construct the biorthogonal family fixed $T>0$. Let $N'\geq 2$. We define $$T':=\min\big\{T,M^2\gamma^{-1}\big\},\ \ \ \ \ \  \ \ C_{N',T'}:=\frac{T'}{2\sum_{k=N'}^\infty\frac{1}{k^3}},$$  $$ a_k:=\frac{C_{N',T'}}{k^2},\ \ \ \ \ \  \ \ P_{N',T'}(z):=e^{iz\frac{T'}{2}} \sum_{k=N'}^\infty \cos(a_k z).$$ 
\begin{lemma}\label{lemma_theta}
The function $P_{N',T'}$ is entire over $\C$ and satisfies 
\begin{itemize}
\item $P_{N',T'}(0)=1$;
\item for every $z\in\C$ such that $Imm(z)\geq 0$, we have $|P_{N',T'}(z)|\leq 1 $;
\item for every $z\in\C$, we have $|e^{-i\frac{T'}{2}}P_{N',T'}|\leq e^{|z|\frac{T'}{2}}$.
\end{itemize}
Moreover, there exists $\theta_0>0$ and $\theta_1>0$, both independent of $N'$ and $T'$, such that $P_{N',T'}$ satisfies \begin{itemize}
    \item if $\Big(\frac{C_{N',T'}|x|}{\theta_0}\Big)^{\frac{1}{2}}+1\geq N'$, then $\ln|P_{N',T'}(x)|\leq -\frac{\theta_1}{2^3}\Big(\frac{C_{N',T'}|x|}{\theta_0}\Big)^\frac{1}{2}$;
        \item if $\Big(\frac{C_{N',T'}|x|}{\theta_0}\Big)^{\frac{1}{2}}+1\leq N'$, then $\ln|P_{N',T'}(x)|\leq -\frac{\theta_1}{N'^3}\Big(\frac{C_{N',T'}|x|}{\theta_0}\Big)^{2}$.
\end{itemize}
\end{lemma}
\begin{proof}
See the proof of \cite[Lemma 4.3]{bio_primo}.
\end{proof}
Under the validity of Lemma \ref{lemma_theta}, we can consider now $N'$ such that
$$N'\geq 2+\frac{M^2\theta}{\gamma^2 T'},\ \ \ \ \ \  \ \theta:=2^{11}\pi^2 \frac{\theta_0}{\theta^2_1}. $$
In order to construct the biorthogonal family of Proposition \ref{prop-bio-fam}, we introduce the following functions 
$$\forall\, m\in\N,\ \forall\, z\in \C,\ \ \ f_{m,N',T'}(z)=F_m(z)\frac{P_{N',T'}(-z)}{P_{N',T'}(i\lambda_m)}.$$
Now, each $f_{m,N',T'}$ verifies \cite[Lemma 4.4]{bio_primo} and we define the Fourier transform of $f_{m,N',T'}(-x) e^{i x \frac{T}{2}}$:
$$\psi_{m,N',T'}(\xi):=\frac{1}{2\pi}\int_\R f_{m,N',T'}(-x)e^{-ix\frac{T}{2}}e^{-ix\xi}dx.$$
We use the functions $\psi_{m,N',T'}$ to build the biorthogonal family as explained in the following Lemma. 

\begin{lemma}\label{lemma_penultimo}
Let $T>0$. Let $T'$ and $N'$ as above. The functions
$$\sigma^+_{m,N',T'}(t)=\psi_{m,N',T'}\Big(\frac{T}{2}-t\Big)e^{-\lambda_m T},\quad \forall\,m\in\N^*$$
form a biorthogonal family to the exponentials $\{e^{-\lambda_m t}\}_{m\in\N^*}$ verifying the inequality \eqref{biort-bound}.
\end{lemma}
\begin{proof}
The proof follows exactly the one of \cite[Lemma 4.5]{bio_primo} by using the inequality provided in Lemma \ref{biort4}.
\end{proof}

\begin{proof}[Proof of Proposition \ref{prop-bio-fam}]
The claim follows directly from Lemma \ref{lemma_penultimo} by noticing that the upper bounds provided for the elements of the biorthogonal family do not depend on the intermediate parameters $N'$ and $T'$. \end{proof}

\end{document}